\documentclass[12pt,reqno,draft]{amsproc}
\usepackage[cp1251]{inputenc}

\usepackage{epstopdf}

\usepackage{amsmath}
\usepackage{amsthm}
\usepackage{amsfonts,amssymb}
\usepackage{color}
\usepackage{enumerate}
\usepackage{graphicx}
\usepackage{enumerate}

\theoremstyle{plain}
\newtheorem{thm}{Theorem}[section]
\newtheorem*{thmA}{Theorem~A}
\newtheorem{cor}{Corollary}[section]
\newtheorem{lem}{Lemma}[section]

\theoremstyle{remark}

\theoremstyle{plain}
\numberwithin{equation}{section}

\def\va{\varepsilon}
\def\C{} \renewcommand\C{\mathbb{C}}
\newcommand\N{\mathbb{N}}
\newcommand\R{\mathbb{R}}
\renewcommand\S{\mathbb{S}}

\newcommand\B{\mathbb{B}}

\newcommand\cE{\mathcal{E}}

\newcommand\cL{\mathcal{L}}

\renewcommand\Im{\operatorname{Im}}
\newcommand\sign{\operatorname{sign}}
\newcommand\supp{\operatorname{supp}}

\newcommand\Cdot{\mathop\cdot}

\def \s{\sigma}
\def \sph{\mathbb{S}^d}
\def\RR{\mathbb{R}}
\newcommand\HH{\mathcal{H}}

\DeclareMathOperator{\dist}{dist}
\def\f{\frac}
\def\o{\omega}

\def\SS{\mathbb{S}}
\def\la{{\langle}}
\def\ra{{\rangle}}

\def\NN{\mathbb{N}}

\newcommand\wh{\widehat}

\def\va{\varepsilon}

\def\sub{\substack}

\def\Bl{\Bigl} \def\Br{\Bigr}

\def\bl{\bigl}
\def\br{\bigr}

\def\sa{\sigma}
\def\al{{\alpha}}
\def\a{{\alpha}}
\def\be{{\beta}}

\def\ld{\lambda}

\def\s{{\sigma}}

\def\va{\varepsilon}

\def\ZZ{{\mathbb Z}}
\def\Ga{\Gamma}
\def\CC{\mathbb{C}}

\begin{document}

\title[]{Estimates of the asymptotic  Nikolskii constants for~spherical polynomials}

\author{Feng Dai}
\address{F.~Dai, Department of Mathematical and Statistical Sciences\\
University of Alberta\\ Edmonton, Alberta T6G 2G1, Canada.}
\email{fdai@ualberta.ca}

\author{Dmitry~Gorbachev}
\address{D.~Gorbachev, Tula State University,
Department of Applied Mathematics and Computer Science, 300012 Tula, Russia.}
\email{dvgmail@mail.ru}

\author{Sergey Tikhonov}
\address{S.~Tikhonov, Centre de Recerca Matem\`atica\\
Campus de Bellaterra, Edifici~C 08193 Bellaterra (Barcelona), Spain; ICREA, Pg.
Llu\'is Companys 23, 08010 Barcelona, Spain, and Universitat Aut\`onoma de
Barcelona.}
\email{stikhonov@crm.cat}

\thanks{F.~D. was supported
by NSERC Canada under the grant RGPIN 04702 Dai. D.~G. was supported by the
Russian Science Foundation under grant 18-11-00199. S.~T. was partially
supported by MTM 2014-59174-P and 2014 SGR 289.
The authors would like to thank the Isaac Newton Institute for Mathematical Sciences, Cambridge, for support and hospitality during the programme Approximation, sampling and compression in data science
 where work on this paper was partially undertaken. This work was also supported by EPSRC grant no EP/K032208/1.}


\begin{abstract}
Let $\Pi_n^d$ denote the space of spherical polynomials of degree at most $n$
on the unit sphere $\mathbb{S}^d\subset \mathbb{R}^{d+1}$ that is  equipped with the
surface Lebesgue measure $d\sigma$  normalized by $\int_{\mathbb{S}^d} \,
d\sigma(x)=1$. This paper  establishes a close connection between   the
asymptotic Nikolskii constant,
$$ \mathcal{L}^\ast(d):=\lim_{n\to \infty} \frac 1 {\dim \Pi_n^d} \sup_{f\in
\Pi_n^d} \frac { \|f\|_{L^\infty(\mathbb{S}^d)}}{\|f\|_{L^1(\mathbb{S}^d)}},$$ and
the following   extremal problem:
$$ \mathcal{I}_\alpha:=\inf_{a_k} \Bigl\| j_{\alpha+1} (t)-  \sum_{k=1}^\infty  a_k
j_{\alpha} \bigl( q_{\alpha+1,k}t/q_{\alpha+1,1}\bigr)\Bigr\|_{L^\infty(\mathbb{R}_+)}
$$
with the infimum being taken over all sequences $\{a_k\}_{k=1}^\infty\subset
\mathbb{R}$ such that the infinite series converges absolutely a.e. on $\mathbb{R}_+$. Here
$j_\alpha $ denotes the Bessel function of the first kind normalized so that
$j_\alpha(0)=1$, and $\{q_{\alpha+1,k}\}_{k=1}^\infty$ denotes the strict increasing
sequence of all positive zeros of $j_{\alpha+1}$. We prove  that for $\alpha\ge
-0.272$,
$$\mathcal{I}_\alpha=
\frac{\int_{0}^{q_{\alpha+1,1}}j_{\alpha+1}(t)t^{2\alpha+1}\,dt}{\int_{0}^{q_{\alpha+1,1}}t^{2\alpha+1}\,dt}=
{}_{1}F_{2}\Bigl(\alpha+1;\alpha+2,\alpha+2;-\frac{q_{\alpha+1,1}^{2}}{4}\Bigr).
$$
As a result, we deduce that the constant $\mathcal{L}^\ast(d)$ goes to zero
exponentially fast as $d\to\infty$:
\[
0.5^d\le \mathcal{L}^{*}(d)\le (0.857\cdots)^{d\,(1+\varepsilon_d)}   \     \     \   \   \    \text{with $\varepsilon_d =O(d^{-2/3})$.}
\]

  \end{abstract}

\date{\today}
\maketitle

\section{Introduction}

Let $\SS^{d}=\{x\in \mathbb{R}^{d+1}\colon |x|=1\}$ denote the unit sphere of
$\RR^{d+1}$, and $d\s$  the  surface Lebesgue  measure on $\sph$ normalized by $\int_{\sph} d\s(x)=1$,  where $|\Cdot|$ denotes the Euclidean norm of $\RR^{d+1}$.  Denote by $\o_d:=\f {2\pi^{\f {d+1}2}}{\Ga(\f {d+1}2)}$  the surface area of $\sph$.
Given $0<p\le \infty$, let $L^p(\SS^{d})$ denote  the
Lebesgue $L^p$-space defined with respect to the measure $d\sigma$ on
$\SS^{d}$, and $\|\Cdot\|_p=\|\Cdot\|_{L^{p}(\SS^{d})}$ the { quasi-norm} of
$L^p(\SS^{d})$.  We denote by $R_k^{(\al,\be)}$, $\al,\be\in\RR$, the usual Jacobi polynomial of degree $n$  normalized by $R_k^{(\al,\be)}(1)=1$, and by $C_k^\mu$, $\mu>0$, the  Gegenbauer polynomial of degree $k$.

 A spherical polynomial of degree at most $n$ on $\sph$ is the restriction to $\sph$  of an algebraic  polynomial
in $d+1$ variables of total degree at most $n$.
Let $\Pi_n^d$ denote the space of all spherical polynomials of
degree at most $n$ on $\SS^{d}$.  As is well known (see, e.g.,
\cite[Chap.~1]{DX13}),
\begin{equation}\label{1-1-19}
\dim \Pi_{n}^{d}=\frac{2n+d}{n+d}\binom{n+d}{n}=
\frac{2n^{d}}{\Gamma(d+1)}\,(1+O(n^{-1})),\quad
n\to\infty.
\end{equation}

The   classical Nikolskii inequality for spherical polynomials (\cite{Ka84}) asserts that  there exists a positive  constant $C_d$ depending only on the dimension $d$ such that
for any  $0<p<q\leq \infty$,
\[
\|f\|_{L^q(\sph)} \leq C_d \bl( \dim \Pi_n^d\br)^{\f 1p-\f 1q} \|f\|_{L^p(\sph)},\quad \forall\,f\in \Pi_n^d.
\]
In this paper, we are mainly interested in the  best Nikolskii constant defined as follows:
\begin{equation}\label{1-5}
\mathcal N(\SS^d; n)_{p,q}:=\sup\Bigl\{\|f\|_{L^q(\SS^{d})}
\colon\ f\in\Pi_n^{d}\quad
\text{and}\quad \|f\|_{L^p(\SS^{d})}=1\Bigr\},
\end{equation}
where  $0<p<q\leq \infty$ and $n\in\mathbb{N}$.

Exact values of the constants $\mathcal N(\SS^d; n)_{p,q}$ are known only in the case of $p=2$ and $q=\infty$, where one has  (see \cite{DGT18})
\begin{equation}\label{1-5-0}
\mathcal{N}(\SS^d; n)_{2,\infty} =\sqrt{\text{dim}\ \Pi_n^d}.
\end{equation}
For the  general case,  the following estimates are  known (\cite{Ka84, DGT18}):
\begin{equation}\label{Nikol}
0<c_d< \f{\mathcal{N}(\SS^d; n)_{p,q}}{(\dim \Pi_n^d)^{\frac 1p-\frac 1q}}\leq C_d<\infty,\quad 0<p<q\leq \infty,
\end{equation}
where $C_d=1$  in the case of $0<p\leq 2$.   However,
 it is a long-standing open problem  to find the exact values of the Nikolskii  constants
$\mathcal N (\sph; n)_{p,q}$  for $(p,q)\ne (2,\infty)$ and $0<p<q\le
\infty$.   In fact, this problem is open even in the case of $d=1$  (\cite{AD14, Go05}).


For $d=1$,
Levin and
Lubinsky \cite{LL15a, LL15b}  established very  close connections between the asymptotic behaviour of the quantity  $\f{\mathcal{N}(\SS^1; n)_{p,q}}{(2n+1)^{\f 1p-\f1q}}$ as $n\to \infty$  and   the best  Nikolskii constant for  entire functions of exponential type  on $\RR$. Their results were recently extended to the higher-dimensional case  by the current authors \cite{DGT18}.
To state these results more precisely,   we   recall   that
an entire function $f$ of $d$-complex variables is  said to be of  spherical exponential type at most $\sa>0$ if for every $\va>0$ there exists a constant $A_\va>0$ such that
$|f(z)|\leq A_\va e^{(\sa+\va)|z|}$ 
for all $z=(z_1, \cdots, z_d)\in\CC^d$. Given $0<p\leq \infty$, we denote by  $\mathcal{E}_{p}^{d}$ the class of all entire functions $f$ in $d$-variables
of  spherical exponential type at most  $1$  whose restrictions to $\RR^d$  belong to  the space $L^p(\RR^d)$ 
(\cite[Ch.~3]{Ni75}).
For $0<p<q\leq \infty$, define
\[
\mathcal{N}(\RR^d)_{p,q}:=\sup\Bl\{\|f\|_{L^q(\RR^d)}\colon\ f\in\mathcal{E}_{p}^{d}\ \
\text{and}\ \ \|f\|_{L^p(\RR^d)}=1\Br\}.
\]
{Throughout this paper, we will consider the Nikolskii constants for real-valued functions (i.e., those functions in $\mathcal{E}_p^d$ whose restrictions to $\mathbb{R}^d$ are real-valued). This will not cause any problem as for every
  $f\in\mathcal{E}_p^d$,
$g(z):=\frac12\,(f(z)+\overline{f(\bar{z})})$ is a function in   $\mathcal{E}_p^d$ whose restriction to $\RR^d$ is real-valued   (see also \cite[Theorem
1.1]{GT17}).}

The following result was  proved first  by   Levin and Lubinsky \cite{LL15a,
LL15b} for $d=1$ and  later by the current authors \cite{DGT18} for $d\ge
2$.\;\footnote{Note that  the definition of the  constant
$\mathcal{N}(\sph,n)_{p,q}$ here is slightly different from that of the
constant  $C(n,d,p,q)$  in \cite{DGT18}  due to the normalization of the
surface Lebesgue measure $d\sa$.
   Indeed, we have  $\mathcal{N}(\sph; n)_{p,q} =\o_d^{\f1p-\f1q} C(n,d,p,q)$. }


\begin{thmA}[\cite{LL15a, LL15b, DGT18}] 
     For $0< p<\infty$, we have
    \begin{equation}\label{1-6-0}
  \lim_{n\to \infty} \frac{\mathcal{N}(\sph; n)_{ p,\infty}} {(\dim\Pi_n^d)^{\f1p}}=\Bl(\f {(2\pi)^d}{V_d}\Br)^{\f1p}\mathcal{N}(\RR^d)_{p, \infty}=:  \mathcal{L}^\ast_{p,\infty}(d),
    \end{equation}
    where $V_d:=\f{\pi^{d/2}}{\Ga(\f d2+1)}$ denotes the volume of the unit ball in $\RR^d$. Furthermore,  if  $0< p<q\le \infty$, then
    \[
      \liminf_{n\to \infty} \frac{\mathcal {N}(\sph; n)_{p,q}}{(\dim\Pi_n^d)^{\f1p-\f1q}}\ge \Bl(\f {(2\pi)^d}{V_d}\Br)^{\f1p-\f1q}
    \mathcal{N}(\RR^d)_{p,q}=:\mathcal{L}^\ast_{p,q}(d).
    \]
\end{thmA}

Note that \eqref{1-6-0} implies that
\begin{equation}\label{1-8}
\mathcal{N}(\sph; n)_{p,\infty} = \mathcal {L}^\ast_{p, \infty} (d) \bl(\dim \Pi_n^d\br)^{\f1p}  \bl( 1+o(1)\br),\quad 0<p<\infty,\quad \text{as $n\to\infty$. }
\end{equation}
Furthermore, by \eqref{1-5-0}, \eqref{1-1-19} and \eqref{1-8}, we obtain
\[
\mathcal{N}(\RR^d)_{2,\infty}=
\f {\sqrt{V_d}} {(2\pi)^{d/2}}=
\Bigl(\frac{1}{2^d \Ga (\f d2+1) \pi^{d/2}}\Bigr)^{1/2}.
\]

In this paper,  we  continue the research of  \cite{DGT18}.  We shall establish more explicit duality formulas for the constants $\mathcal{N}(\sph; n)_{p,\infty}$ and $\mathcal{N}(\RR^d)_{p,\infty}$ with $1\leq p<\infty$.  For example,  in Section~\ref{sec-dual}, we prove
\begin{equation}\label{1-8-b}
\mathcal{N}(\sph; n)_{1,\infty} =(\dim \Pi_n^d) \inf_{a_k}  \Bl\| R_n^{(\frac d2, \f {d-2}2)}-\sum_{k=n+1}^\infty a_k C_k ^{\f {d-1}2}\Br\|_{L^\infty [-1,1]}
\end{equation}
with the infimum being taken over all sequences $\{a_k\}_{k=n+1}^\infty \subset \RR$ such that the series $\sum_{k=n+1}^\infty a_k C_k^{\f {d-1}2}(t)$ converges  to an essentially bounded function  in $L^2$-norm with respect to the measure $(1-t^2)^{\f {d-2}2} dt$ on  $[-1,1]$.
 One of our main  goals  is to apply these duality formulas to  estimate  the constant $\mathcal{L}^\ast_{p,\infty} (d)$ in the asymptotic expansion \eqref{1-8}  for $p=1$.
For simplicity, we~write $$\mathcal{L}^\ast (d):
=\mathcal{L}^\ast_{1,\infty}(d).$$
Note that  by \eqref{1-5-0}  and \eqref{Nikol},   if   $  0<p\leq 2$,  then for any $d$,  \begin{equation}\label{1-8-00}
\cL^\ast_{p,\infty}(d)\leq 1
\end{equation} with equality for $p=2$.

The estimates for the Nikolskii constant $\cL^\ast (d)$ are important in many applications. Let us mention only a few of them here.  First of all,  the constant
$\mathcal{L}^\ast(d)$  appears very naturally in problems on
best $L^1$-approximation (see, e.g., \cite{BKP12, Ge38,GV14,LS17}).  It can  be
used to obtain certain  tight bounds in the Remez-type problem about the
concentration of $L^{1}$-norm of entire functions of the spherical exponential
type (\cite{BKP12,MR14}, see also \cite{TT17}).
 Some  details can be found  in  Section \ref{sec-remez}.
The next example is widely known.    The  constant $\cL^\ast (d)$ can be used   to
obtain some lower   tight-bounds for spherical designs (see   \cite{Lev98}).
  Moreover, the Nikolskii constants play an important role in approximation of  smooth,  multivariate  functions  defined  on  irregular  domains
 by polynomial  frame  approximation method \cite{ben}.
  More detailed historical comments on the constant $\cL^\ast (d)$ and related background information
  will be given in
 Section~\ref{sec-back}.

While  it remains to be very challenging to find the   exact values of the constants $\cL^{*}(d)$,
in \cite{DGT18} we
solved this problem  for non-negative functions in the class $\mathcal{E}_1^d$:
\begin{thm}[\cite{DGT18}] \label{thm-1-2}For $d\in \mathbb{N}$,  we have
  \[
  \cL^{+}(d) :=\f{(2\pi)^d}{V_d}
  \sup_{\sub{f\in\mathcal{E}_1^d\setminus \{0\},\\
      f\text{$\br|$}_{\RR^d}\ge 0}} \f{\|f\|_{L^\infty(\RR^d)}
  }{\|f\|_{L^1(\RR^d)}}=2^{-d}.
  \]
\end{thm}

One of the main results in this paper  asserts   that   the estimate \eqref{1-8-00}  can be significantly improved  for $p=1$, and the constant $\cL^\ast(d)$ 
goes to zero  exponentially fast as $d\to\infty$:

\begin{thm}\label{thm-main-as}
  For $d\in\N$, we have
  \[
  2^{-d}\le \cL^{*}(d)\le {}_{1}F_{2}\Bigl(\frac d2;\frac d2+1,\frac d2+1;
  -\frac{\beta_{d}^{2}}{4}\Bigr)=
  \f{\int_{0}^{\beta_d}j_{d/2}(t)t^{d-1}\,dt}{\int_0^{\be_d} t^{d-1}\, dt},\]
  where
  $_{1}F_{2}$  denotes the usual hypergeometric function,
  { $j_{d/2}$ is the normalized Bessel function},
  and $\beta_{d}=q_{d/2, 1}$ is  the smallest positive
  zero of the Bessel function $J_{d/2}$ of  the first kind.
\end{thm}

\begin{cor}\label{cor-1-1}
  For $d\ge 2$, we have
  \[
  2^{-d}\le \cL^{*}(d)\le (\sqrt{2/e})^{d\,(1+\va_d)},
  \]
  where $\sqrt{2/e}=0.857\cdots$, and $\va_d =O(d^{-2/3})$ as $d\to\infty$.
\end{cor}

    Using Theorem \ref{thm-main-as},
    we may  obtain  the  numerical  upper estimates  of $\cL^\ast(d)$
    for $d=1,2,\ldots,10$, listed in the following table:

    \begin{table}[h]
      \begin{tabular}{|c|c|c|c|c|c|c|c|c|c|c|}
        \hline
        $d$    & 1 & 2 &3 &4  &5&6&7&8&9&10\\
        \hline
        \text{upper bounds}    &  0.589 & 0.382& 0.261& 0.184& 0.133& 0.098& 0.073 &0.055& 0.042 &0.032\\
        \hline
      \end{tabular}
    \end{table}
    Note that for $d=1$,   we  recover the  upper bound  of $\cL^\ast(1)$ previously  obtained   in
    \cite{HB93,AKP96}, while  for $d=2$,  our method with more delicate calculations leads to the following estimate:
    \[
    \mathcal{N}(\S^2; n)_{1,\infty} =\cL^{*}(2)n^{2}(1+o(1))\quad \text{with }\quad \cL^{*}(2)\in
    (0.2820,0.3822),
    \]
    which improves the corresponding known estimate in
    \cite{AZ76, Ho76}.

Let us  give a few comments on the proof of Theorem \ref{thm-main-as}.  Clearly, the lower estimate in Theorem \ref{thm-main-as}  follows directly from Theorem \ref{thm-1-2}. However,
the proof of   the upper estimate in Theorem \ref{thm-main-as} is much more involved. It  relies on   a duality argument (see, for instance,  \eqref{1-8-b}).   The crucial ingredient in the proof is to solve  an extremal problem on   $L^\infty$-approximation  by the  Bessel functions of the first kind on $\R_+=[0,\infty)$, which seems to be of independent interest.

To be more precise, we need to introduce several notations.  For $\al\in\CC$, let $J_\al$ denote the Bessel function of the first kind, and $j_\al$   the normalized Bessel function given by
\[
j_{\alpha}(z):=2^{\alpha}\Gamma(\alpha+1)\,\frac{J_{\alpha}(z)}{z^{\alpha}},\  \ z\in\C.
\]
Let $\{q_{\a, k}\}_{k=1}^\infty$ denote the strictly  increasing  sequence of all positive zeros of  $j_\a(z)$.
For $\al>-\f12$, we denote by $X_\al$ the set of all functions $F\in L^\infty [0,\infty)$ that  can be represented as
an   infinite sum of the  form
\[
F(t) :=\sum_{k=1}^\infty  a_k j_{\al} \bl( r_{\al+1,k}t\br),\quad t\ge 0,\quad a_k\in\RR,\quad r_{\al+1,k} =\f {q_{\al+1,k}}{q_{\al+1,1}},\quad k\in\NN.
\]
Here we assume that the series converges absolutely to $F$ almost everywhere on $\RR_+$.
In the proof of the upper estimate in  Theorem  \ref{thm-main-as} , we are required to solve the following extremal problem for $\al=\f d2-1$:
\begin{equation}\label{dual1}
\mathcal{I}_\al:=  \inf_{F\in X_\al}    \|j_{\al+1}-F\|_{L^\infty(\R_+)}.
\end{equation}
In this paper, we find the  exact value of $\mathcal{I}_\al$  for $\al\ge -0.272$,  from which the upper estimate in Theorem  \ref{thm-main-as} will follow:

\begin{thm}\label{thm-5-1}
  Let $\al\ge -0.272$ and let $\mathcal{I}_\al$ be defined in \eqref{dual1}. Then
  \begin{equation}
  \mathcal{I}_\al= {}_{1}F_{2}\Bigl(\alpha+1;\alpha+2,\alpha+2;-\frac{q_{\al+1,1}^{2}}{4}\Bigr)  =  \frac{\int_{0}^{q_{\al+1,1}}j_{\alpha+1}(t)t^{2\alpha+1}\,dt}{\int_{0}^{q_{\al+1,1}}t^{2\alpha+1}\,dt}. \label{5-2-0}
  \end{equation}
\end{thm}

The identity  \eqref{5-2-0} extends the following   result 
(see \cite{AKP96,Go05}):
\begin{equation}\label{akp}
  \inf_{a_k\in\RR} \Bigl\|\frac{\sin t}{t}-\sum_{k=1}^{\infty}a_{k}\cos
    kt\Bigr\|_{L^\infty(\R_+)}=\frac1\pi\int_0^\pi \frac{\sin x}x\,dx.
\end{equation}
    We point out that the proof of this last formula in \cite{Go05}  relies on the fact that the corresponding extremal function is a periodic function, which does not seem to work in our situation. Our proof of \eqref{5-2-0} in this paper is different from that in \cite{Go05}.

This paper is organized as follows. Section 2 contains some background  information and historical comments on sharp Nikolskii constants.  Some
preliminary materials on spherical harmonics and Bessel functions are given in Section~\ref{sec:2}. In
Section~\ref{sec-dual},  we deduce more explicit duality formulas for the
Nikolskii constants, and connect our problem with several other extremal
problems in approximation theory.  We also  study  the existence, uniqueness
and characterizations of the  corresponding extremal functions for these
extremal problems in Section~\ref{sec-dual}.
After that, in Section~\ref{sec:5}, we prove the  main theorem,
Theorem~\ref{thm-5-1}, from which  the upper estimates in
Theorem~\ref{thm-main-as} will follow. The proof of
Corollary~\ref{cor-1-1} is given in Section~\ref{sec:6}.   Finally, in
Section~\ref{sec-remez}, we show how our results on the Nikolskii constants can be used to deduce certain interesting  Remez-type results.

\section{Historical background}\label{sec-back}

{In this  section, we give some background information and historical
comments  on the Nikolskii constants. Nikolskii inequalities have been playing
crucial roles in  approximation theory and harmonic analysis, particularly in
the embedding theory of function spaces (see \cite{Ni75, DT05}).
%
%
%

In the case of $d=1$,  the problem of finding  the  exact values of the
constants $\mathcal{N}(\mathbb{S}^1; n)_{1,\infty}$   has a very long history,
starting with the work of  Jackson  \cite{Ja33} in 1933. A  closed form of  the
constant $\mathcal{N}(\S^1; n)_{1,\infty}$,  which  is not very useful in
applications,  was found by Geronimus \cite{Ge38}. Stechkin (see \cite{Ta65,
Ta93}) {proved that there is a constant~$c$ 
such that $\mathcal{N} (\S^1; n)_{1,\infty} = cn +o(n)$ as $n\to \infty$, while
Taikov \cite{Ta65}) further proved  that $\mathcal{N} (\S^1; n)_{1,\infty} = c
n +O(1)$ with  $c\in (0.539,0.584)$. In \cite{Go05,GM18} it was established
that $c=\mathcal{L}^\ast(1)$ and for any $n$ and $0<p<\infty$
\[
(2n)^{1/p}\cL^\ast_{p,\infty}(1)\le \mathcal{N}(\S^1; n)_{p,\infty}\le
(2n+2\lceil p^{-1}\rceil)^{1/p}\cL^\ast_{p,\infty}(1)
\]
(see also \cite{LL15a,GT17}).} 
In the limiting case of  $p=0$,
Arestov \cite{Ar80} found the exact values of the Nikolskii constants  for  the
trigonometric polynomials  on the unit circle $\mathbb{S}^1$. Finally, in the
case of $d\ge 2$ and  $0<p<q=\infty$,    Arestov and Deikalova \cite{AD14}
   proved  that the supremum in \eqref{1-5} can be achieved by   zonal
polynomials, and as a result, the Nikolskii constant  $\mathcal{N}(\S^2; n)_{1,\infty}$  for spherical polynomials coincides with the Nikolskii constant
for algebraic polynomials in $L^{1}([-1,1])$ \cite{AZ76, Ho76}.

{As  was mentioned in the introduction,  of  crucial importance  in the
proofs of the main results in this paper are  the  duality formulas for the Nikolskii constants, which will be given in the next section.  In the case of
$\SS^{1}$ this approach was introduced by Taikov \cite{Ta65}, who  established the
classical Bernstein  result on  the best approximation of $\cos nx$ by
functions $\sum_{k=n+1}^{\infty}a_{k}\cos kx\in L^{\infty}[0,2\pi)$. L.~H\"ormander and B.~Bernhardsson \cite{HB93}  proved that
\begin{equation}\label{hb}
\cL^\ast (1)=\inf_{v}\,\Bigl\|\frac{\sin
x}{x}-v(x)\Bigr\|_{L^{\infty}(\mathbb{R})},\quad
\widehat{v}=0\quad \text{in}\quad (-1,1),
\end{equation}
 and described general properties of
the extremal function $G\in \mathcal{E}_{1}^{1}$ satisfying $\cL^\ast
(1)=\frac{\|G\|_{L^{\infty}(\mathbb{R})}}{\|G\|_{L^{1}(\mathbb{R})}}$.  Furthermore, they also computed the following very precise
numerical value:   $\cL^\ast (1)\approx 0.54092882$ (cf. with
\cite{Go05,GM18}).

In the particular case when $v$ has the form $\sum_{k=1}^{\infty}a_{k}\cos kt$
the problem \eqref{hb} was considered by Andreev, Konyagin, and Popov \cite{AKP96}
(see \eqref{akp}), who studied a constant that is equivalent to $\cL^\ast (1)$
via the Fourier transform, that is (see also \cite{Go05})
\begin{equation}\label{akp-}
\cL^\ast (1)=\sup_{F\ne 0}\frac{|F(0)|}{\|F\|_{L^{1}(\mathbb{R})}},\quad
\widehat{F}=0\quad \text{in}\quad [-1,1]^{c}.
\end{equation}

Some interesting applications of the Nikolskii constants  in  number theory can be found in the paper by Carneiro, Milinovich, and Soundararajan \cite{CMS18}, who considered
a family of problems related to the Nikolskii constant \eqref{akp-} and applied the resulting
estimates  to study the problem on the  distribution of prime numbers.
The paper \cite{CMS18} also  considers  a version of the Nikolskii problem when $\widehat{F}\le
0$ outside $[-1,1]$. This problem for $F\ge 0$ corresponds to  the extremal
Cohn-Elkies problem (also called the Delsarte problem) that is connected with the problem of sphere
packing (see, e.g., \cite{Go00a,CE03,CKMRV17,Vi17}).

We also refer to \cite{Bo54,Ni75,NW78,BKP12,AD14,GT17,ABDH18} for more background
information on classical Nikolskii constants. }

}

\section{Preliminaries}\label{sec:2}

In this section, we present  some preliminary materials on spherical harmonics and Bessel functions,  most of which  can be found in  \cite{DX13}, \cite[Chap.~7]{BE53},
\cite{OLBC10}, and \cite{Wa66}.

First, a spherical harmonic of degree $n$ on $\sph$ is the restriction to $\sph$ of a  homogeneous harmonic polynomial in $d+1$ variables of total degree $n$. We denote by
 $\HH_n^d$  the space of all
spherical harmonics of degree $n$ on $\SS^{d}$.
As is well known, the spaces $\HH_n^d$, $n=0,1,\cdots$,  are mutually orthogonal with respect to the
inner product of $L^2(\SS^{d})$, and for each non-negative integer $n$, the function $\f {k+\ld} {\ld }\,C_k^\ld(x\cdot y)$, $x, y\in\sph$ is the reproducing kernel of the space $\HH_n^d$, where $ \ld=\f{d-1}2$ and $C_n^\lambda$ denotes  the usual  Gegenbauer polynomial of degree $n$,  as defined in \cite{Sz67}. Thus,
\[
f(x) =\frac{k+\lambda}{\lambda}\int_{\SS^{d}} f(y) C_k^\lambda (x\Cdot
y)\,d\sigma(y),\quad x\in\SS^{d},\quad f\in\HH_k^d.
\]
As a result,  the reproducing kernel of the space $\Pi_n^d$ of spherical polynomials of degree at most $n$ on $\sph$ is given by
\begin{equation}\label{reproduce}
G_{n}(x\cdot y):=\sum_{k=0}^n \f{k+\ld}\ld\,C_k^{\ld}(x\cdot y)= (\dim \Pi_n^d)R_n^{(\frac{d}{2},\frac{d-2}{2})}(x\cdot y),
\end{equation}
where $R_n^{(\al,\be)}$ denotes the  normalized Jacobi polynomial of degree $n$:  $$R_{n}^{(\alpha,\beta)}(t)=\frac{P_{n}^{(\alpha,\beta)}(t)}{P_{n}^{(\alpha,\beta)}(1)}.$$

Second,
an entire function $f$ of $d$-complex variables is of spherical exponential type at most $\sa$ if for every $\va>0$ there exists a constant $A_\va>0$ such that
$|f(z)|\leq A_\va e^{(\sa+\va)|z|}$
for all $z=(z_1, \cdots, z_d)\in\CC^d$ (see \cite[Chap.~3]{Ni75}). Given $0<p\leq \infty$,  denote by  $\mathcal{E}_{p}^{d}$   the class of all entire functions
of  spherical exponential type at most  in $d$-variables whose restrictions to $\RR^d$  belong to  the space $L^p(\RR^d)$.
If  $0<p<q\leq \infty$, then  $\mathcal{E}_p^d\subset \mathcal{E}_q^d$
and there exists a constant $C=C_{d,p,q}$ such that $\|f\|_q \leq C \|f\|_p$ for all $f\in\mathcal{E}_p^d$.
Moreover,
every function $f\in \mathcal{E}_{p}^{d}$ is bounded on
$\mathbb{R}^{d}$ and satisfies  $|f(z)|\leq \|f\|_{L^{\infty}(\mathbb{R}^{d})}e^{\sa |\!\Im
(z)|}$, $\forall\,z\in\CC^d$.
According to the Palay-Wiener theorem,   each  function  $f\in \mathcal{E}_{p}^{d}$ can be identified with a function  in $ L^p(\RR^d)$ whose
distributional Fourier transform is  supported in the unit ball $\mathbb{B}^{d}:=\{x\in \RR^{d}\colon
|x|\le 1\}$.  Here we recall that
the  Fourier transform of $f\in L^1(\RR^d)$ is defined by
\[
\mathcal{F} _df(\xi)\equiv \wh{f}(\xi)=\int_{\RR^d} f(x) e^{- i x\Cdot \xi}\, dx,\quad \xi\in\RR^d,
\]
while   the inverse Fourier transform is given by
\[
\mathcal{F}_d^{-1} f(x)=\frac 1{(2\pi)^d}\int_{\RR^d} f(\xi) e^{ix\Cdot \xi}\, d\xi,\quad f
\in L^1(\RR^d),\quad x\in\RR^d.
\]

Finally, we present  some well-known properties of  the Bessel functions,  most of which  can be found in   \cite[Chap.~7]{BE53},
 and \cite{Wa66}.   The Bessel function $J_\al$ of the first kind is the solution to the differential equation
\begin{equation}\label{ODE}
x^2 y'' +x y'+ (x^2-\al^2)y=0
\end{equation}
such that the limit  $\lim_{x\to 0} x^{-\al} J_\al(x)$ exists and is finite.
Denote by $j_\al$  the normalized Bessel function given by
$$
j_{\alpha}(z):=2^{\alpha}\Gamma(\alpha+1)\,\frac{J_{\alpha}(z)}{z^{\alpha}},\  \ z\in\C.$$
As is well known,
$j_\al(z)$ is an even entire function of exponential type $1$  satisfying that  $j_\al(0)=1$ and
  \begin{align}
(x^{2\alpha+2}j_{\alpha+1}(x))'&=(2\alpha+2)x^{2\alpha+1}j_{\alpha}(x),\quad j_{\alpha}'(x) =-\frac{xj_{\alpha+1}(x)}{2\alpha+2}.\label{bess-deriv}
\end{align}
Moreover,
  \begin{align}
  |j_\al(x)|\leq C (1+|x|)^{-\al-\f12},\quad x\in\RR.\  \label{2-4}
  \end{align}
Note that
\eqref{bess-deriv}  also implies (see  \cite[7.2.8 (56)]{BE53})
\begin{equation}\label{2-6-00}
-\f{z j_{\al+2}(z)}{2(\al+2)}
=j_{\alpha+1}^{\,\prime}(z)=
\frac{2(\alpha+1)}{z}\left(j_{\alpha}(z)-j_{\alpha+1}(z)\right). \end{equation}

  If $\al=\f d2-1$, then the function $j_{\al}(|\Cdot|)$ is the Fourier transform of the normalized  surface Lebesgue measure on the sphere  $\S^{d-1}$, while  if $\al=\frac d2$, then the function $\f {V_d} {(2\pi)^d}\,j_{\al}(|\Cdot|)$ is the Fourier transform of the characteristic function $\chi_{\B^d}$ of the unit ball $\B^d$. That is,
\[
j_{\frac d2-1}(|\xi|)=\int_{\S^{d-1}}e^{-ix\Cdot
  \xi}\,d\sa(x)=\wh{\sa_{d-1}}(\xi),\quad \xi\in\RR^d,
\]
and
\begin{equation}\label{2-7}
\f {V_d} {(2\pi)^d}\,\mathcal{F}_d\Bl( j_{\f d2} (|\Cdot|)\Br)(\xi)=\chi_{\B^d}(\xi),\quad \xi\in\RR^d,
\end{equation}
where the Fourier transform
$\mathcal{F}_d$ is understood in a distributional sense or in the space of $L^2(\R^d)$.

The zeros of $j_\al (z)$ are all simple and real. Let $\{q_{\a, k}\}_{k=1}^\infty$ denote the sequence of all positive zeros of  $j_\a(z)$ arranged so that  $0<q_{\alpha,1}<q_{\alpha,2}<\dots$. For convenience, we also set $q_{\al,0}=0$.
Then
  $q_{\alpha,k}\sim \pi k$ as
$k\to \infty$, and for   $\alpha>0$, the smallest positive zero of $j_\al(z)$ satisfies
\begin{equation}\label{2-9-2}
\sqrt{\alpha(\alpha+2)}<q_{\alpha,1}<\sqrt{\alpha+1}\,(\sqrt{\alpha+2}+1).
\end{equation}
Moreover, \begin{equation}\label{2-10-19}
j_{\alpha}(z)=\prod_{k=1}^{\infty}
\Bigl(1-\frac{z^{2}}{q_{\alpha,k}^{2}}\Bigr),\quad z\in\CC.
\end{equation}

The  positive zeros of   $j_{\alpha}(z)$ and $j_{\alpha+1}(z)$ are interplaced:
\begin{equation}\label{interorder}
0<q_{\alpha,1}<q_{\alpha+1,1}<q_{\alpha,2}<q_{\alpha+1,2}<\dots.
\end{equation}

The following result  on the zeros of the Bessel functions will be used repeatedly in later sections:

\begin{lem}
For $
\alpha>-1/2$ and $k=1,2,\cdots,$ we have
\begin{equation}\label{8-10-0}
\max_{z\ge q_{\al+1,k}} |j_{\alpha}(z)|= |j_{\alpha}(q_{\alpha+1,k})|=(-1)^kj_{\alpha}(q_{\alpha+1,k})>0.
\end{equation}
\end{lem}
\begin{proof}
  For the sake of  completeness, we include a short  proof of this lemma here.
  Since $$\bl( j_\al(z)^2\br)'=2j_\al(z) j_\al'(z)=-\f { z j_{\al+1} (z) j_\al(z)}{\al+1},$$
   the function $(j_{\alpha}(z))^2$ achieves its local maxima on $(0,\infty)$ at the positive zeros of $j_{\al+1}(z)$,  on which  we also have
  $(j_{\alpha}(z))^2=(j_{\alpha}(z))^2+(j_{\alpha}^{\,\prime}(z))^2$.
  However, it is easily seen from  \eqref{ODE}  that
  \begin{equation*}
  \Bl((j_{\alpha}(z))^2+(j_{\alpha}^{\,\prime}(z))^2
  \Br)'=
  -\frac{2(2\alpha+1)}{z}\,
  (j_{\alpha}^{\,\prime}(z))^2,
  \end{equation*}
  which implies that the function  $(j_{\alpha}(z))^2+(j_{\alpha}^{\,\prime}(z))^2$
  is strictly  decreasing on $(0,\infty)$ if  $\alpha>-1/2$.
  Thus, the sequence
  $$\Bl\{\Bl(j_\al(q_{\al+1,k})\Br)^2\Br\}_{k=1}^\infty=
  \Bl\{\Bl(j_\al(q_{\al+1,k})\Br)^2+ \Bl\{\Bl(j_\al'(q_{\al+1,k})\Br)^2\Br\}_{k=1}^\infty$$
  is strictly increasing.  It then follows that
  \[
  \max_{z\ge q_{\al+1,k}} |j_{\alpha}(z)|= |j_{\alpha}(q_{\alpha+1,k})|.
  \]
  Finally, the second equality in \eqref{8-10-0} is a direct consequence of \eqref{interorder}.
\end{proof}

For $\al>-\f12$,
the  Fourier--Bessel expansion of a function $f\in L^1 ([0,1], t^{2\al+1}\, dt)$  with respect to the orthogonal basis $\{ j_\al ( q_{\al+1,k} x)\}_{k=0}^\infty$ is given by
\begin{equation}\label{bessj-sum}
f(t)=\sum_{k=0}^{\infty}h_k^{-1}c_{k}(f)j_{\alpha}(q_{k}t),\quad t\in [0,1],
\end{equation}
where
\begin{align*}
h_{0}&=\int_0^1 t^{2\al+1}\, dt=\frac{1}{2\alpha+2},\quad h_{k}=\int_{0}^{1}j_{\alpha}^{2}(q_{k}t)t^{2\al+1}\,dt=
\frac{j_{\alpha}^{2}(q_{k})}{2},\quad k=1,2,\cdots, \\
c_{k}(f)&=\int_{0}^{1}f(t)j_{\alpha}(q_{k}t)t^{2\al+1}\,dt,\quad k=0,1,\cdots.
\end{align*}  If  $f\in C^{1}([0,1])$, then the series \eqref{bessj-sum} converges absolutely outside of a neighbourhood of the origin.

For later applications, we also  record here the following two useful  formulas on  Bessel functions ( see \cite[Sect.~6.2.10]{Lu69}):

\begin{align}
\int_{0}^{1}j_{\alpha}(at)j_{\alpha}(bt)t^{2\al+1}\,dt&=
\frac{a^{2}j_{\alpha+1}(a)j_{\alpha}(b)-
b^{2}j_{\alpha}(a)j_{\alpha+1}(b)}{2(\alpha+1)(a^{2}-b^{2})},\quad a>b>0,\label{2-13-0}\\
\int_{0}^{z}t^{2\alpha+1}j_{\alpha+1}(t)\,dt&=
\frac{z^{2\alpha+2}}{2\alpha+2}\,
{}_{1}F_{2}\Bigl(\alpha+1;\alpha+2,\alpha+2;-\frac{z^{2}}{4}\Bigr),\quad z>0. \label{2-14-0}
\end{align}

\section{Duality  formulas  and characterizations of certain  extremal functions  }\label{sec-dual}

 The main goals  in this section are  to prove  some duality     formulas for the Nikolskii constants $\mathcal{N}(\sph; n)_{p,\infty}$ and $ \mathcal{N}(\RR^d)_{p,\infty}$, and to characterize the corresponding extremal functions in the dual spaces.
 These results   will play an important role in the proofs of  our main theorems in the next section.
 For simplicity, we shall write $\mathcal{N}(\sph; n)_p=\mathcal{N}(\sph; n)_{p,\infty}$ and $\mathcal{N}(\RR^d)_p=\mathcal{N}(\RR^d)_{p,\infty}.$

We start with some necessary notations. Let $w_d(t)=c_d (1-t^{2})^{d/2-1}$, where $c_d>0$ is a normalization constant such that $\int_{-1}^1 w_d(t)\, dt =1$. For $1\leq p\leq \infty$,  we denote by   $L^p([-1,1]; w_d)\equiv L^p(w_d)$   the usual  Lebesgue $L^p$-space defined with respect to the measure $w_d(t)\, dt $  on $[-1,1]$, and $\|\cdot\|_{L^p(w_d)}$ the Lebesgue $L^p$-norm of the space $L^p(w_d)$.
  Denote by $\mathcal{P}_n$ the space of all univariate algebraic polynomials of degree at most $n$.
  Define
    \[
  \mathcal{P}_n^{\bot} = \Bigl\{ F\in L^{1}(w_{d})\colon\ \ \int_{-1} ^1 F(t) t^j w_d(t)\, dt =0,\quad j=0,1,\cdots,n\Bigr\},
  \]
  and $  \mathcal{P}_{n,p}^{\bot} =  \mathcal{P}_n^{\bot}  \cap L^{p}(w_d)$ for $1\leq p\leq \infty$.
  Finally, given a normed linear space $(X,\|\cdot\|)$, the distance of  a vector $x\in X$ from a set $E\subset X$ is defined by
  $$\dist(x, E)_X:=\inf_{y\in E}\|x-y\|.$$
  As is well known, if $E$ is a linear subspace of $X$, then one has (see, for instance,   \cite[p.~61,  Theorem~1.3]{DL}),
 \begin{equation}\label{3-1-19}
 \dist(x, E)_X=\max_{\sub{\ell\in E^{\bot}\\
     \|\ell\|_{X^\ast} \leq 1}}  |\la \ell, x\ra|,
 \end{equation}
where $$E^{\bot}:=\Bl\{ \ell \in X^\ast,\quad \la \ell, y\ra =0,\quad \forall\,y\in E\Br\}$$ and $X^\ast$ denotes the dual of $X$.

Next, recall  that $j_\al(z) = \Ga(\al+1) (t/2)^{-\al} J_\al(t)$ denotes the normalized Bessel function of the first kind.
Let  $K(|x|):=\f {V_d} {(2\pi)^d}\,j_{d/2}(|x|)$.
For convenience, we will  use a slight abuse of the notation that  $f(x)=f(|x|)$ for a radial function on $\R^d$.
By \eqref{2-7}, $\mathcal{F}_d K(\xi)=\wh{K}(\xi)=\chi_{\B^d}(\xi)$ for every $\xi\in\RR^d$, and by \eqref{2-4},  $K(|\Cdot|)\in L^{q}(\R^{d})$ for
$q>\frac{2d}{d+1}$.
It then follows that for each
$1\le p<\frac{2d}{d-1}$,
\begin{equation}\label{f-K}
f(x)=\int_{\R^d} f(y)K(|x-y|)\, dy,\quad x\in\R^d,\quad f\in\cE_p^d.
\end{equation}
In particular, for a radial function $f(|\Cdot|) \in\cE_p^d$ with $1\le p<\frac{2d}{d-1}$,
\[
f(0)=\int_{0}^{\infty}K(t)f(t)\,v_{d}(t)\,dt,
\]
where $v_{d}(t):=\omega_{d-1}t^{d-1}$.
Let  $L^p(v_d)$ denote the Lebesgue $L^p$-space defined with respect to the measure $ v_d(t) dt$ on $[0,\infty)$. Clearly,
for each $f\in L^p(v_d)$, $\|f(|\Cdot|)\|_{L^p(\RR^d)} =\|f\|_{L^p(v_d)}$.


Our duality results for the Nikolskii constants on the sphere can be stated as follows:

\begin{thm}\label{thm-dual-S} If $1\le p<\infty$  and $\f 1p+\f 1{p'}=1$, then  for every positive integer $n$,
  \begin{equation}\label{3-3-1}
  \mathcal{N}(\sph; n)_p=\dist(G_n, \mathcal{P}_{n,p'}^{\bot})_{L^{p'}(w_d)},
  \end{equation}
  where $G_n$ is   the reproducing kernel of the space $\Pi_n^d$  given in  \eqref{reproduce}.
Moreover,    there exists  a minimizer $F_\ast\in \mathcal{P}_{n,p'}^{\bot}$ of  the form
\begin{equation}\label{3-4-0}
F_\ast=G_n-  \f { P^{*}(1)|P^{*}|^{p-1}\sign P^{*}}{\|P^\ast\|_{L^p(w_d)}^p},
\end{equation}
such that $\|G_n-F_\ast\|_{L^{p'}(w_d)} =\dist(G_n, \mathcal{P}_{n,p'}^{\bot} )_{L^{p'}(w_d)},$
where
  $P^\ast$ denotes the unique  algebraic polynomial of degree $n$  such that
  \[
  \|P^\ast\|_{L^p(w_\ast)}=\dist(x^n, \mathcal{P}_{n-1})_{L^p(w_\ast)}\quad \text{with
  \ ${w}^\ast_d (t) =w_d(t) (1-t)$}.
  \]

\end{thm}

 Before stating  the similar duality results on $\RR^d$,  we first  note that
\begin{equation}\label{extrem}
\mathcal{N}(\RR^d)_{p} =\sup\Bl\{ |f(0)|\colon\ \ f\in\mathcal{E}_p^d,\quad \|f\|_p=1\Br\},\quad 1\leq p<\infty.
\end{equation}
This holds because each $f\in\mathcal{E}_p^d$ achieves its maximum on $\RR^d$
{(due to the fact that $f(x)\to 0$ as $|x|\to \infty$ \cite[3.2.5]{Ni75})}
and the space $\mathcal{E}_p^d$ is invariant under the usual translations
on~$\RR^d$.

Duality formulas for functions in $\cE_p^d$ can now be stated as follows:

\begin{thm}\label{prop-f*}   The following statements hold:
  \begin{enumerate}[\rm (i)]
    \item For $1\leq p<\infty$, there exists an unique radial  extremizer
  $f_{*}\in \cE_{p}^{d}$ for the supremum in \eqref{extrem}  such that $\|f_{*}\|_{L^p(\RR^d)}=1$ and
  $f_{*}(0)=\mathcal{N}(\RR^d)_p$.
  Furthermore, such an extremizer can be characterized via the following identity:
    \begin{equation}\label{f-f*}
  g(0)=f_\ast(0)\int_{\R^{d}}g(x)|f_\ast(|x|)|^{p-1}\sign f_\ast(|x|)\,dx,\quad \forall\,g\in
  \cE_{p}^{d};
  \end{equation}
that is,  a    radial function $ f_\ast(|\Cdot|)\in \mathcal{E}_p^d$  with $\|f_\ast\|_{L^p(v_d)} =1$  is an extremizer   for  \eqref{extrem} if and only if  the condition \eqref{f-f*} is satisfied.

\item If $1\leq p<\f {2d}{d-1}$ and $\f 1p+\f 1{p'}=1$, then
\[
\mathcal{N}(\RR^d)_p=\f {V_d}{(2\pi)^d}\dist(j_{d/2}, \mathcal{E}_{p'}^{\bot})_{L^{p'} (v_d)},
\]
where  $\mathcal{E}_{p'}^{\bot}$ denotes the space of all functions $f\in L^{p'}(v_d)$ such that
$$\int_0^\infty f(t) g(t) v_d(t)\, dt =0\quad \text{whenever $g(|\Cdot|) \in \cE_p^d$}.$$

\item For  each  $1\leq p<\f {2d}{d-1}$,  there exists an
unique
extremizer $F_\ast \in \mathcal{E}_{p'}^{\bot}$, which takes the form
\begin{equation}\label{3-10-0}
F_\ast (t) = \f {V_d }{(2\pi)^d}j_{d/2}(t) -f_\ast(0) |f_\ast (t)|^{p-1} \sign f_\ast (t),\quad t\ge 0,
\end{equation}
 such that
\[
\Bl  \|F_\ast -
\f {V_d }{(2\pi)^d}\,j_{d/2}\Br\|_{L^{p'} (v_d)}=\f {V_d}{(2\pi)^d}\dist(j_{d/2}(|\Cdot|), \mathcal{E}_p^{\bot})_{L^{p'} (v_d)}.
\]
Here $f_\ast$ denotes the extremal function in (i).
     \end{enumerate}
\end{thm}

As pointed out in the introduction, the main goal in this paper is to estimate the following normalized
Nikolskii constant for $p=1$: $$\cL^{*}(d)=(2\pi)^{d}V_{d}^{-1}\mathcal N(\R^d)_1.$$
By  the
Paley--Wiener--Schwarz theorem \cite{NW78},  if  $f\in \mathcal{E}_1^d$,  then $\supp \wh{f}\subset \B^{d}$, and hence for any $r\ge 1$,
\begin{align*}
0&=\int_{\S^{d-1}}\wh{f}(r\xi)\,d\sa(\xi)
=(\wh{D_r f})\ast (d\sa) (0)\\
&=\int_{\R^d} (D_r f)(-x) j_{\f d2-1}(|x|) \, dx =\int_{\RR^d} f(x) j_{\f d2-1} (r|x|)\, dx,
\end{align*}
where $D_r f(x) =r^{-d} f(x/r)$, and in the third step we used the fact that  the distributional Fourier transform of $j_{d/2-1}(|\Cdot|)$ is the normalized Lebesgue measure $d\sa$ on $\SS^{d-1}$. This shows that $j_{\f d2-1} (r|\Cdot|) \in \cE_\infty^{\bot}$ for any $r\ge 1$, as desired.

Since
$j_{\f d2-1}(r \cdot) \in \cE_\infty^{\bot}$ for any $r\ge 1$, by  Theorem \ref{prop-f*}  we obtain

\begin{cor}\label{cor-L*}
  For $d\ge 2$, we have
  \[
\cL^\ast(d)\leq   \inf_{a_k\in\RR, r_k\ge 1} \Bigl\|j_{\frac
    d2}(\cdot)-\sum_{k=1}^{\infty}a_{k}j_{\frac
    d2-1}(r_{k}\cdot )\Bigr\|_{L^{\infty}(\R)}
  \]
  with  the infimum being taken over all sequences $\{a_k\}_{k=1}^\infty\subset \R$ and $\{r_k\}_{k=1}^\infty \subset [1,\infty)$ such that $\sum_{k=1}^\infty a_k j_{\f d2-1} (r_k \cdot)$
   converges absolutely  to an essentially  bounded function on $[0,\infty)$.
\end{cor}

\subsection{Proof of Theorem \ref{thm-dual-S}}
 For simplicity, we write $d\mu_d(t)= w_d(t)\, dt$ and $d\mu^\ast_d(t)= w^\ast_d(t)\, dt$ in the proof below.

We start with the proof of \eqref{3-3-1}.
Using orthogonality of spherical harmonics,  we have that for any $f\in\Pi_n^d$ and $x\in\sph$,
\[
f(x)=\int_{\S^{d}}(G_{n}(x\cdot y)-F(x\cdot y))f(y)\,d\s(y),\quad \forall\,F\in \mathcal{P}_n^{\bot}\cap L^{p'}(w_d).
\]
It follows by  H\"older's inequality  that
\begin{align*}
\mathcal{N}(\sph,n)_p &=\sup_{0\neq f\in\Pi_n^d} \f {\|f\|_\infty}{\|f\|_p} \leq
\inf\bigl\{\|G_{n}-F\|_{L^{p'}(w_d)}\colon    F\in \mathcal{P}_n^{\bot}\cap L^{p'}(w_d)
\bigr\}\\
&=\dist(G_n, \mathcal{P}_{n,p'}^{\bot})_{L^{p'}(w_d)}.
\end{align*}

To show the lower estimate $$\mathcal{N}(\sph,n)_p\ge \dist(G_n, \mathcal{P}_{n,p'}^{\bot})_{L^{p'}(w_d)},$$
we  use the duality formula \eqref{3-1-19} with $E:=\mathcal{P}_{n,p'}^{\bot}\subset L^{p'}(w_d)$.  Here, if  $p=1$,  then we use $C[-1,1]$  in place of  $L^\infty$, and recall that
 $(C[-1,1])^\ast$  is the space of Radon measures on $[-1,1]$ with the norm given by the total variation of a measure.    Since $G_j\subset E$ for any $j>n$,  it follows that  $E^{\bot}=\mathcal{P}_n$. Thus, using \eqref{3-1-19} , we obtain
\begin{align}
\dist(G_n, \mathcal{P}_{n,p'}^{\bot})_{L^{p'}(w_d)} &=\sup_{\sub{\ell\in (\mathcal{P}_{n,p'}^{\bot})^{\bot}\\
  \|\ell\|\leq 1 }} |\la \ell, G_n\ra|=\sup_{\sub{\|P\|_{L^{p}(w_d)}\leq 1\\
    P\in \mathcal{P}_n}} \Bl|\int_{-1}^1 P(t)G_n(t) d\mu_d(t)\Br|\notag\\
&=\sup_{\sub{\|P\|_{L^{p}(w_d)}\leq 1\\
      P\in \mathcal{P}_n}} |P(1)|
  \leq \mathcal{N}(\sph; n)_p.\label{3-5-0}
\end{align}
This proves \eqref{3-3-1}.

Next, we show  the existence of the extremal function $F_\ast$ and the formula \eqref{3-4-0}.
The proof relies on  the following  characterization of best approximants in $L^p$-spaces.

\begin{lem}[{\cite[4.2.1, 4.2.2]{Sh71}}]\label{lem-dual}
  Let $Y$ be a closed  real subspace of $L^{p}(Q,d\mu)$ for some measure space $(Q, \mu)$ and $1\le
  p<\infty$.  Let  $f\in L^{p}(d\mu)$ . If      $p=1$, we assume in addition that  $f(x)\neq 0$ for $\mu$-a.e. $x\in Q$.
  Then  a function $g\in Y$ is the best approximant to $f$ from the space $Y$ in $L^p$-metric (i.e.,
  $\|f-g\|_p =\dist(f, Y)_p$)  if and only if
  \[
  \int_Q \Bl(|f-g|^{p-1}\sign
  (f-g)\Br) h \, d\mu =\int_Q \f {|f-g|^p} {f-g}\,h \, d\mu=0,\quad \forall\,h\in Y.
  \]
\end{lem}

Now we continue the proof of Theorem \ref{thm-dual-S}.  Note first that by \eqref{3-3-1} and \eqref{3-5-0},
\begin{equation*}
\mathcal{N}(\sph; n)_p =\max\Bl\{ P(1)\colon\ \ P\in\mathcal{P}_n,\  \ \|P\|_{L^p(w_d)}\leq 1\Br\}.
\end{equation*}
Let $P_\ast\in\mathcal{P}_n$ denote  the maximizer for the maximum  in this last equation. Then $\|P_\ast\|_{L^p(w_d)} =1$ and
\[
P(1) \leq P_\ast (1) \|P\|_{L^p(w_d)},\quad \forall\,P\in \mathcal{P}_n.
\]
In particular, this implies that for any
$$ \mathcal{P}_{n,0}:=\{ P\in \mathcal{P}_n\colon\ \ P(1)=0\},$$
we have
\begin{align*}
 P_\ast (1)&\leq P_\ast (1) \inf_{P\in\mathcal{P}_{n,0}} \|P_\ast-P\|_{L^p(w_d)}=P_\ast (1) \dist(P_\ast, \mathcal{P}_{n,0})_{L^p(w_d)}\\
 &\leq P_\ast(1) \|P_\ast\|_{L^p(w_d)}=P_\ast(1).
\end{align*}
Thus,
$$1= \|P_\ast\|_{L^p(w_d)} =  \dist(P_\ast, \mathcal{P}_{n,0})_{L^p(w_d)}.$$
It then follows from  Lemma \ref{lem-dual} that
\begin{equation}\label{3-10}
\int_{-1}^1 (|P_\ast|^{p-1} \sign (P_\ast)) P\,d\mu_d=P(1) \int_{-1}^1 |P_\ast|^{p-1} \sign (P_\ast)\,d\mu_d,\quad \forall\,P\in\mathcal{P}_n.
\end{equation}
Setting $P=P_\ast$ in \eqref{3-10},
we obtain
\begin{equation}\label{3-8}
1=\|P_\ast\|_{L^p(w_d)}^p =P_\ast (1) \int_{-1}^1 |P_\ast|^{p-1} \sign (P_\ast)\,d\mu_d.
\end{equation}
Multiplying both sides of \eqref{3-10} by $P_\ast (1)$ and using \eqref{3-8}, we have
\[
P_\ast(1) \int_{-1}^1 (|P_\ast|^{p-1} \sign (P_\ast)) P d\mu_d=P(1) =\int_{-1}^1 P(t) G_n(t) d\mu_d (t),\quad \forall\,P\in\mathcal{P}_n.
\]
This implies that
$$ F_\ast(t):=G_n(t)-P_\ast(1) |P_\ast(t)|^{p-1} \sign (P_\ast(t))
\in\mathcal{P}_{n,p'}^{\bot}.$$
Note also that
$$ \|G_n-F_\ast\|_{L^{p'}(w_d)} =P_\ast (1)\|P_\ast\|_{L^p(w_d)}^{p-1} =P_\ast (1) =\dist(G_n, \mathcal{P}_{n,p'})_{L^{p'}(w_d)}.$$
This shows  \eqref{3-4-0} and   that $F_\ast$ is the desired extremal function.

Finally, we  point out that the connection of $P_\ast$ with the extremal polynomial $P^\ast$ was proved in \cite{AD14}. For completeness, we include a proof of the identity $P_\ast =P^\ast/\|P^\ast\|_{L^p(w_d)}$ here.
By \eqref{3-10}, we have
\begin{align*}
\int_{-1}^1 |P_\ast|^{p-1}\sign(P_\ast) \f {P(t)-P(1)}{1-t} w_d^\ast (t)\, dt =0,\quad \forall\,P\in\mathcal{P}_{n},
\end{align*}
or equivalently,
\begin{align}\label{3-13}
\int_{-1}^1 |P_\ast|^{p-1}\sign(P_\ast) P(t)  w_d^\ast (t)\, dt =0,\quad \forall\,P\in\mathcal{P}_{n-1}.
\end{align}
By Lemma \ref{lem-dual}, this implies that
$$\|P_\ast\|_{L^p(w_\ast)} =\dist(P_\ast, \mathcal{P}_{n-1})_{L^p(w^\ast)}=|L_{n, P_\ast}| \dist(x^n, \mathcal{P}_{n-1})_{L^p(w_d^\ast)},$$
where $L_{n, P_\ast}$ denotes the leading coefficient of the $n$-th degree polynomial $P_\ast$.
Note that by \eqref{3-13}, we have $\deg(P_\ast) =n$ and all the zeros of $P_\ast$ are simple and inside the interval $(-1,1)$. Since $P_\ast (1)>0$, we must  have $L_{n, P_\ast}>0$.
It then follows that
$$P_\ast = L_{n, P_\ast} P^\ast =\f {P^\ast}{\|P^\ast\|_{L^p(w_d)}}.$$

This completes the proof of  Theorem \ref{thm-dual-S}.
\subsection{Proof of Theorem \ref{prop-f*}}

We start with the proof of (i), which  relies on the following compactness  result on entire functions of exponential type.

\begin{lem}[{\cite[3.3.6]{Ni75}}]\label{lem-comp} Let $1\leq p<\infty$ and let $\mathcal{B}_p^d:=\{f\in \cE_p^d\colon\ \|f\|_p\leq 1\}$.  Then every sequence of functions from the class $\mathcal{B}_p^d$  contains a subsequence which converges uniformly to a function $f\in\mathcal{B}_p^d$ on every compact subset of $\RR^d$. \end{lem}

By \eqref{extrem},  there exists a sequence of functions $\{f_{n}\}\subset \mathcal{B}_p^d$ such that $\lim_{n\to\infty}  f_{n}(0)=\mathcal{N}(\RR^d)_p$.   By  Lemma~\ref{lem-comp},  without loss of generality, we may assume that $\{f_n\}_{n=0}^\infty$ converges uniformly to a function $f\in\mathcal{B}_p^d$ on every compact subset of $\RR^d$ (since otherwise we consider a subsequence of $\{f_n\}$). Then
  \[
  f(0)=\lim_{n\to \infty} f_{n}(0)=\mathcal{N}(\RR^d)_p.
  \]
  Now consider the following  radial part of the function $f$:
  \[
  f_{*}(|x|):=\int_{\S^{d-1}}f(|x|\xi)\,d\sa(\xi),\quad x\in\R^d.
  \]
  For convenience, we will identify $f_\ast$ with the radial function $f_\ast(|\Cdot|)$ on $\RR^d$.  It is easily seen that $f_{*}\in \cE_p^d$  and $f_{*}(0)=f(0)=\mathcal{N}(\RR^d)_p$ . Thus, $f_{*}$ is
  the extremal function in \eqref{extrem}; that is,
  \begin{equation}\label{3-16}
    |f(0)|\leq f_\ast(0) \|f\|_p,\quad \text{for every $f\in \cE_p^d$}.
  \end{equation}
  The proof of the characterization \eqref{f-f*}  of $f_\ast$ follows exactly as that for spherical polynomials on the sphere. Indeed,  applying \eqref{3-16} to $f=f_\ast -h$ with $g\in\cE_p^d$ and $g(0)=0$, we obtain
  $$\|f_\ast\|_p =\dist(f_\ast, \mathcal{H})_p,$$
  where
  $\mathcal{H}:=\{f\in \cE_p^d\colon\ f(0)=0\}.$  By Lemma \ref{lem-dual}, this implies that $|f_\ast|^{p-1}\sign (f_\ast) \bot \mathcal{H}$, which is equivalent to \eqref{f-f*}.
  That \eqref{f-f*} implies that $f_\ast\in\mathcal{E}_p^d$ is the desired extremal function follows directly from H\"older's inequality.

  Next, we show the uniqueness of $f_\ast$.  For $1<p<\infty$, the uniqueness follows directly from  the  strict convexity of the space $L^p$.
    It remains to deal with the case $p=1$.
    We consider $f_\ast$ as a  function in one variable so that  $f_\ast $ is an even  entire function of exponential type  on $\CC$.  By  the classical
  Hadamard theorem (see, e.g., \cite[Ch.~2]{Bo54}), it follows that
  $$f_\ast (z) = f_\ast(0) e^{bz} \prod_{n\in\ZZ\setminus \{0\}}  \Bl (1-\f {z}{z_n}\Br) \exp\Bl( \f z{z_n}\Br),$$
  where $b\in\CC$ and $\{z_n\}_{n\in\ZZ\setminus \{0\}}$ is the sequence of all nonzero zeros of $f_\ast$.
  Since $f_\ast$ is an even function, we may assume that
  $z_n=z_{-n}$ for all positive integers $n$.  Thus, we must have that $b=0$, and
  \begin{equation}\label{Hadamard}
    f_{*}(z)=f_{*}(0)\prod_{n=1}^{\infty}\Bigl(1-\frac{z^{2}}{z_{n}^{2}}\Bigr),\quad
  z\in \C.
  \end{equation}
We  further claim that all the zeros of $f_\ast$ must be simple and real (i.e., the numbers $z_n$ are distinct and real).  To see this, we first recall that $f_\ast\br|_{\RR}$ is real valued, which implies
$\overline{f_\ast (z)}=f_{\ast}(\bar{z})$ for every $z\in \CC$.
Thus, if the claim were not true, then for a nonzero complex zero $w$ of $f$, we have the decomposition
$$ f(z) = (z^2-w^2) (z^2-\bar{w}^2) g(z)=a(z)g(z),$$
where $g\in \cE_1^1$, and $a(z)= z^4-(w^2+\bar{w}^2)z^2 +|w|^4$.   It is easily seen that $a(x)\ge 0$ for all $x\in\RR$ and the  equality holds only in the case when $x=w\in\RR$. Thus,  $\sign(f_\ast)(x) =\sign(g)(x)$ for all $x\in\RR\setminus \{w\}$. Now consider the functions
$$f_t (z) = a(tz) g(z),\quad z\in\CC,\quad t\in\RR.$$
Clearly, $f_t(|\Cdot|) \in \cE_1^d$ and $f_\ast(0) =f_t(0)>0$ for all $t\in\RR$. Thus, \eqref{f-f*} implies
$$ 1=\f {f_t(0)}{f_\ast (0)} =\int_{\RR^d} \sign(f_\ast) (|x|) f_t(|x|) \, dx =\int_{\RR^d} a(t|x|) |g(|x|)|\, dx,\quad \forall\,t\in\RR.$$
This last term in this last equation  is a polynomial in $t\in\RR$ of degree $4$, which can not be constant.  We obtain a contradiction and hence  prove the claim.

  Now assume that $f_{\ast\ast}$ is another radial  extremizer for the supremum in \eqref{extrem}.  Then  \eqref{f-f*} implies
  \[
  \|f_{\ast\ast}\|_1=1=\frac{f_{\ast\ast}(0)}{f_{*}(0)}=\int_{\R^{d}}f_{\ast\ast}(x)\sign
  f_{*}(x)\,dx.
  \]
It follows that
   $|f_{\ast\ast}(x)|= f_{\ast\ast}(x)\sign f_{*}(x)$  for a.e.  $x\in \RR$.   Since all the zeros of $f_\ast$ are simple, it follows that $f_\ast$ changes signs at each of its zeros.  By continuity, this further implies that $|f_{\ast\ast}(x)|=f_{\ast\ast}(x) \sign f_\ast (x)$ for every $x\in\RR$.    By symmetry, we also have $|f_{*}|\equiv
  f_{*}\sign f_{\ast\ast}$.   This means that $f_{\ast\ast}$ and $f_{*}$ have common zeros. By \eqref{Hadamard} and the above claim, we conclude that  $f_{\ast\ast}\equiv
  f_{*}$, proving the uniqueness.

We point out that the proofs of the duality formulas \eqref{3-3-1} and \eqref{3-4-0}  are very similar to those  of \eqref{3-3-1} and \eqref{3-4-0} for
spherical polynomials on $\sph$.  We skip the details.

 Finally,  we show the uniquenss of the extremal function  $F_{*}$ defined by \eqref{3-10-0}.    For
 $p> 1$, the uniqueness follows directly of  strict convexity of the $L^p$-norm. It remains to consider the case of $p=1$.

Recall that $K(|\Cdot|)=\f {V_d} {(2\pi)^d}\,j_{d/2}(|\Cdot|)$,
$f_{*}(0)=\mathcal{N}(\RR^d)_1$, and $\|f_{*}\|_{L^{1}(\mathbb{R}^{d})}=1$. If
$F$ is extremal then $\|K-F\|_{L^{\infty}(\mathbb{R})}=f_{*}(0)$. By
\eqref{f-K} and \eqref{f-f*}, we have
\[
\int_{\R^{d}}(K(|x|)-F(|x|))f_\ast(|x|)\,dx=
f_\ast(0)-\int_{\R^{d}}F(|x|)f_\ast(|x|)\,dx=
f_\ast(0)
\]
On the other hand,
\begin{align*}
f_\ast(0)&=
\int_{\R^{d}}(K(|x|)-F(|x|))f_\ast(|x|)\,dx\le
\int_{\R^{d}}|K(|x|)-F(|x|)|\,|f_\ast(|x|)|\,dx\\
&\le
\|K-F\|_{L^{\infty}(\mathbb{R})}\|f_{*}\|_{L^{1}(\mathbb{R}^{d})}=
f_\ast(0).
\end{align*}
Thus, we have the sharp H\"older inequality for $(p,p')=(\infty,1)$  in the third step. It follows that
\[
K(t)-F(t)=f_\ast(0)\sign f_{*}(t)\quad \text{a.e. for $t\ge 0$}.
\]
Therefore, $F_{*}=K-f_\ast(0)\sign f_{*}$ is unique.

\section{Proof of Theorem \ref{thm-5-1} }\label{sec:5}



%
This section is devoted to the proof of Theorem \ref{thm-5-1}.

  Note first  that  the indentify,
    \begin{equation}\label{def-a*0-0}
    {}_{1}F_{2}\Bigl(\alpha+1;\alpha+2,\alpha+2;-\frac{q_{\al+1,1}^{2}}{4}\Bigr)=
    \frac{\int_{0}^{q_{\al+1,1}}j_{\alpha+1}(t)t^{2\alpha+1}\,dt}{\int_{0}^{q_{\al+1, 1}}t^{2\alpha+1}\,dt},
    \end{equation}
  follows directly from  \eqref{2-14-0}.  It remains to show that
  \begin{equation}\label{5-2-19}
  \mathcal{I}_\al = a_0^\ast:=\frac{\int_{0}^{q_{\al+1,1}}j_{\alpha+1}(t)t^{2\alpha+1}\,dt}{\int_{0}^{q_{\al+1,1}}t^{2\alpha+1}\,dt},
  \end{equation}
  where $\mathcal{I}_\al$ is defined in \eqref{dual1}.

    For simplicity,  we write
    $r_0=q_0=0$, $q_k=q_{\al+1,k}$,  and $r_k=r_{\al+1,k}=\f {q_k}{q_1}$  for $k=1,2,\cdots$.   Recall that $\{j_\al (q_k t)\}_{k=0}^\infty$ forms an orthogonal basis of the space $L^2([0,1], t^{2\al+1}\, dt)$.  In particular, we have
   \begin{equation}\label{orthogonailty}
   0=q_1^{2\al+2}\int_0^1 j_{\al} (q_k t) t^{2\al+1}\, dt =\int_0^{q_1} j_{\al} (r_k t) t^{2\al+1} \, dt,\quad k=1,2,\cdots.
   \end{equation}

  The crucial step  in our proof is to construct  an extremal function  $F^\ast\in X_\al$ with the following properties:
  \begin{equation}\label{5-4-0}
  \sup_{t>0} |j_{\al+1}(t)-F^\ast (t)|\leq a^\ast_0:=
  \frac{\int_{0}^{q_{1}}j_{\alpha+1}(t)t^{2\alpha+1}\,dt}
  {\int_{0}^{q_{1}}t^{2\alpha+1}\,dt},
  \end{equation}
  and
  \begin{equation}\label{5-5-0}
  j_{\al+1} (t) -F^\ast(t)\equiv a_0^\ast,\quad \text{for a.e. $t\in [0, q_1]$}.
  \end{equation}

  For the moment, we  assume that there exists   an extremal function $F^\ast\in X_\al$ satisfying  \eqref{5-4-0} and \eqref{5-5-0},
  and proceed with the proof of \eqref{5-2-19}. Indeed,  by   \eqref{def-a*0-0} and \eqref{5-4-0}, we have
$$\mathcal{I}_\al \leq a_0^\ast= {}_{1}F_{2}\Bigl(\alpha+1;\alpha+2,\alpha+2;-\frac{q_{\al+1,1}^{2}}{4}\Bigr).$$
Thus, if $\mathcal{I}_\al\neq a_0^\ast$, then  $\mathcal{I}_\al< a_0^\ast$, and   there exists a function $F\in X_\al$ such that
$$\|j_{\al+1} -F\|_\infty < a_0^\ast.$$
However, by \eqref{5-5-0}, this implies that
$$ j_{\al+1} (t) -F(t) < j_{\al+1}(t) -F^\ast (t)\quad \text{for a.e. $t\in [0, q_1]$},$$
or equivalently,
$$ F^\ast (t) -F(t)<0\quad \text{for a.e. $t\in [0, q_1]$}.$$
Integrating this last inequality with respect to the measure $t^{2\al+1} dt$ on  $[0, q_1]$, we obtain
$$\int_0^{q_1}( F^\ast (t) -F(t)) t^{2\al+1}\, dt <0,$$
which is impossible since by \eqref{orthogonailty}, we have
$\int_0^{q_1} F(t) t^{2\al+1}\, dt =0$ for every $F\in X_\al$.
Thus, one must have  $\mathcal{I}_\al=a_0^\ast$.

It remains to prove the existence of  an extremal function $F^\ast \in X_\al$ satisfying \eqref{5-4-0} and \eqref{5-5-0}.  Firstly, the condition \eqref{5-5-0} suggests us to consider the Fourier-Bessel series of the function $j_{\al+1}(t)$ with respect to the orthogonal basis $\{ j_{\al} (q_k t/q_1)\}_{k=0}^\infty$ of the space $L^2([0, q_1], t^{2\al+1}dt)$. Indeed, by \eqref{bessj-sum}, we have
  \begin{align}\label{5-7-a}
j_{\alpha+1}(t)=
a^\ast_{0}+\sum_{k=1}^{\infty}a_{k}^\ast j_{\alpha}(r_{k}t),\quad t\in
[0,q_{1}],
\end{align}
where
\[
a^\ast_{0}=
\frac{\int_{0}^{q_{1}}j_{\alpha+1}(t)t^{2\alpha+1}\,dt}
{\int_{0}^{q_{1}}t^{2\alpha+1}\,dt},\quad a_k^\ast=\f 2{(j_\al(q_k))^2} \int_0^1 j_{\al+1} (q_1 t) j_\al(q_k t) t^{2\al+1}\, dt,\quad k\ge 1.
\]
Using \eqref{bess-deriv} , we have
\begin{align*}
 a_k^\ast&=\f {4(\al+1)}{(j_\al(q_k))^2} q_1^{-2\al-2}\int_0^1  j_\al(q_k t) t^{-1} \int_0^{q_1t}  x^{2\al+1}j_\al(x ) \, dx dt\\
 &=\f {4(\al+1)}{(j_\al(q_k))^2} \int_0^1 x^{2\al+1}\int_0^1  j_\al(q_k t) j_\al(q_1 x t) t^{2\al+1} dt\, dx,\quad k\ge 1.
\end{align*}
It then follows by
  \eqref{2-13-0} that
\begin{align}
a^\ast_{k}=&-\frac{2}{j_{\alpha}(q_{k})}
\int_{0}^{1}\frac{j_{\alpha+1}(q_{1}x)x^{2\alpha+3}}{r_{k}^{2}-x^{2}}\,dx,\quad k=1,2,\cdots.\label{ak-def}
\end{align}
Note that $j_{\al+1}(t)>0$ for $t\in [0, q_1)$ and $r_k\ge 1$ for $k\ge 1$. Thus, \eqref{ak-def} together with \eqref{8-10-0} implies that
\begin{equation}\label{5-10-1}
j_\al (q_k) a_k^\ast <0,\quad (-1)^{k+1} a_k^\ast >0,\quad k=1,2,\cdots.
\end{equation}

  Secondly, a straightforward calculation shows that for any $t>0$,
  \begin{equation}\label{5-11}
  |a_k^\ast  j_\al (r_k t)|\leq \f C {k^2} \Bl( \f k {1+kt}\Br)^{\al+\f12},\ \  \text{as $k\to\infty$}.
  \end{equation}
  This in particular implies that \eqref{5-7-a} holds pointwisely  for every $t\in [0, q_1]$. Thus,  by \eqref{5-10-1}, we have
  \begin{equation}\label{a0-sum}
  0=j_{\alpha+1}(q_{1})=a^\ast_{0}+\sum_{k=1}^{\infty}a_{k}^\ast j_{\alpha}(q_{k})=
  a^\ast_{0}-\sum_{k=1}^{\infty}|a_{k}^\ast j_{\alpha}(q_{k})|.
  \end{equation}
  Note also that  \eqref{5-11} implies that  the series in  \eqref{5-7-a}   converges uniformly   on every compact subset of
  $(0,\infty)$ to a function $F^\ast\in L^\infty[0,\infty)$. Thus, we may use the infinite series on the right hand side of \eqref{5-7-a} to
  define a function $F^\ast$ on $\RR_+$  as follows:
  $$ F^\ast(t):=\begin{cases}
  \sum_{k=1}^{\infty}a_{k}^\ast j_{\alpha}(r_{k}t),&\ \ \text{if $t>0$};\\
1-  a^\ast_0, &\  \ \text{if $t=0$}.
  \end{cases}$$
Note that by \eqref{5-7-a},
  \[
  F^\ast(t) +a_0^\ast =j_{\al+1}(t),\quad \forall\,t\in [0, q_1].
  \]
     This together with the uniform convergence of the series on compact subsets of $(0,\infty)$   implies that $F^\ast$ is a  uniformly bounded  continuous function  on $[0,\infty)$.

  %

  Finally, to complete the proof, it remains to verify that
  \begin{equation}\label{j-F-a0}
  |j_{\alpha+1}(t)-F^\ast (t)|\le a^\ast_{0},\quad \forall\,t\ge q_1.
  \end{equation}
    Using \eqref{8-10-0} and \eqref{a0-sum},   we obtain  that for $t\ge q_1$,
  \begin{align*}
  |j_{\alpha+1}(t)-F(t)|
  &\le |j_{\alpha+1}(t)-a^\ast_{1}j_{\alpha}(t)|
  +\sum_{k=2}^{\infty}|a^\ast_{k}j_{\alpha}(q_{k})|\\
  &=a^\ast_{0}-a^\ast_1 |j_\al (q_k)| +|j_{\alpha+1}(t)-a^\ast_{1}j_{\alpha}(t)|.
  \end{align*}
    Thus, for the proof of \eqref{j-F-a0}, it suffices to show that
  \begin{equation}\label{j-a1-j}
\max_{t\ge q_1}\Bl  |j_{\alpha}(t)-\f {j_{\alpha+1}(t)}{a^\ast_{1}}\Br|\le |j_{\alpha}(q_{1})|.
  \end{equation}

  The proof of \eqref{j-a1-j} relies on the following technical lemma, which can be seen as an extension of the property  $\max_{t\ge q_{\al+1,1}} |j_\al (t)|=|j_\al (q_1)|$:
   \begin{lem}\label{lem-6-1} For $\al\ge -0.272$, we have  \begin{equation}\label{5-16-b}
     \sup_{t\ge q_{\al+1,1}}   \Bl|j_{\alpha}(t)-u
   \,j_{\alpha+1}(t)\Br|\equiv  |j_{\alpha}(q_{\al+1, 1})|,\quad 0\leq u\leq \f {\al+2}{\al+1}.
     \end{equation}

   \end{lem}

 The proof of Lemma \ref{lem-6-1} is very technical, and will be given in the next subsection.
   For the moment, we take this lemma for granted and proceed with the proof of \eqref{j-a1-j}.

  By Lemma \ref{lem-6-1}, it is enough to show
  \begin{equation}\label{5-16-1}
  \f 1 {a_1^\ast}\leq \f {\al+2}{\al+1}.
  \end{equation}
  To this end,  we define   $g(x):=\frac{j_{\alpha+1}(q_{1}x)}{1-x^{2}}$ for $x\in [0,1)$.  Then by \eqref{2-10-19},
  \[
  g(x)=
  \prod_{k=2}^{\infty}\Bigr(1-\frac{x^{2}}{r_{k}^{2}}
  \Bigr),\quad x\in [0,1).
  \]
  Since $r_k\ge 1$ for all $k\ge 1$, it is easily seen  that  $g(x)$ is a strictly  decreasing  function on $ [0,1)$.
  Thus, for any $x\in [0,1)$,
  \[
  \frac{j_{\alpha+1}(q_{1}x)}{1-x^{2}}=g(x)>\lim_{x\to 1-} g(x) =\frac{q_{1}j_{\alpha+1}^{\,\prime}(q_{1})}{-2}=
  -(\alpha+1)j_{\alpha}(q_{1}),
  \]
  where the last step uses \eqref{2-6-00}.
  It then follows from   \eqref{ak-def} that
  \[
  a^\ast_{1}=-\frac{2}{j_{\alpha}(q_{1})}
  \int_{0}^{1}\frac{j_{\alpha+1}(q_{1}x)
    x^{2\alpha+3}}{1-x^{2}}\,dx
  >2(\alpha+1)\int_{0}^{1}x^{2\alpha+3}\,dx=\frac{\alpha+1}{\alpha+2},
  \]
    which  proves \eqref{5-16-1}.

\subsection{Proof of Lemma  \ref{lem-6-1}}
Write as usual $q_1:=q_{\al+1,1}$.
We first  claim that for the proof of  Lemma  \ref{lem-6-1},
it is enough to show  \eqref{5-16-b}  for $u=\f {\al+2}{\al+1}$, or equivalently,
\begin{equation}\label{5-18}
\sup_{t\ge q_1} \Bl|j_\al (t)-\f {\al+2}{\al+1}\,j_{\al+1}(t)\Br|=|j_\al (q_1)|.
\end{equation}
To see this,   consider the function
\[
F(t,u):=j_{\alpha}(t)-uj_{\alpha+1}(t),\quad t\ge q_1,\quad  u\ge 0.
\]
 Note that for $t> q_1$ and $u>0$,
\begin{align*}
\nabla F(t,u)& =0\iff \begin{cases}
j'_{\alpha}(t)-uj'_{\alpha+1}(t)=0\\
j_{\al+1} (t)=0
\end{cases}  \iff j_{\al+1}(t)=j_{\al+2}(t)=0,
\end{align*}
which is impossible since  $j_{\al+1}$ and $j_{\al+2}$ do not have common positive zeros.
This means that  $F$ does not have any critical points in the domain $\{(t, u):\  \ t>q_1,\  \ u>0\}$.  On the other hand,  however,  for any $u>0$,
$$\lim_{t\to \infty} \max_{0\leq v\leq u} |F(t,v)|=0.$$
Thus,  for any $u>0$,
 $|F|$ has   a maximum on the domain $D_u:=\{ (t, v):\  \ t\ge q_1,\  \ 0\leq v\leq u\}$  which is achieved  on  its boundary $\partial D_u$.  Since  by  \eqref{8-10-0},
 $$ \sup_{t\ge q_1} |F(t,0)|=\sup_{0\leq v\leq u} |F(q_1,v)|=|j_\al(q_1)|,\quad u>0,$$
 it follows that
 \begin{align*}
 \max_{(t, v)\in D_u} |F(t,v)|= \max_{(t,v)\in \partial D_u} |F(t, v)|=\max_{t\ge q_1}|F(t, u)|=: M(u),\quad u>0.
 \end{align*}
 Since $D_{u_1}\subset D_{u_2}$ for $0\leq u_1<u_2$, this implies that the function $M(u):=\max_{t\ge q_1}|F(t, u)|$ is increasing on $[0, \infty)$.  The claim then follows as $M(u)\ge |F(q_1, u)|=|j_\al(q_1)|$ for any $u\ge 0$.

It remains to show \eqref{5-18}.  Define
\[
f(t):=j_{\alpha}(t)-\frac{\alpha+2}{\alpha+1}\,j_{\alpha+1}(t),\quad t\ge 0.
\]
We need to prove that
\[
\max_{t\ge q_1} |f(t)|=|f(q_1)|=|j_\al(q_1)|.
\]

Using \eqref{2-6-00} and \eqref{bess-deriv}, we obtain
\begin{equation}\label{f'}
f'(t)=-\frac{tj_{\alpha+1}(t)}{2(\alpha+1)}+\frac{tj_{\alpha+2}(t)}{2(\alpha+1)}=
\frac{t^{3}j_{\alpha+3}(t)}{2^{3}(\alpha+1)(\alpha+2)(\alpha+3)}.
\end{equation}
This in particular implies that the local extrema of $f$ on $(0,\infty)$ can  only be attained at positive zeros  of $j_{\al+3}$  (i.e.,  at the points $q_{\al+3,1}$, $q_{\al+3,2}$, $\cdots$).  We claim that
\begin{equation}\label{5-22}
|f(q_{\al+3, k})|\ge |f(q_{\al+3,k+1 })|,\quad k=1,2,\cdots,
\end{equation}
which will imply
\[
\max_{t\ge q_{1}} |f(t)|= \max \Bl\{ |f(q_{1})|, |f(q_{\al+3, 1})|\Br\}.
\]

To show \eqref{5-22}, we need a differential equation for the function $f$.  Indeed, using  \eqref{bess-deriv},  \eqref{f'}, and the formula
$$j_{\alpha}''(t)=-\frac{2\alpha+1}{t}\,j_{\alpha}'(t)-j_{\alpha}(t),$$
we obtain
\begin{equation}\label{5-24}
f''(t)=\frac2t\,j_{\alpha}'(t)-\frac{2\alpha+3}{t}\,f'(t)-f(t).
\end{equation}
Furthermore, using  \eqref{f'} and \eqref{2-6-00},  we can  write $f'$ in the form
\begin{equation}\label{5-25}
f'(t)=
\Bigl(\frac{2(\alpha+2)}{t}-\frac{t}{2(\alpha+1)}\Bigr)j_{\alpha+1}(t)-
\frac{2(\alpha+2)}{t}\,j_{\alpha}(t).
\end{equation}
Now combining \eqref{5-24} with \eqref{5-25},  we deduce via  a~straightforward calculation that
\begin{equation}\label{f-ode}
A_{2}f''+A_{2}f'+A_{0}f=0,
\end{equation}
where
\[
A_{0}=t^{3},\quad A_{1}=(2\alpha+1)t^2+4(\alpha+2)(2\alpha+3),\quad
A_{2}=t(t^2+4(\alpha+2)).
\]

Now let us consider the function $\varphi:=f^{2}+\frac{A_{2}}{A_{0}}\,f'^{2}$.
Using \eqref{f-ode}, and by a~straightforward calculation, we obtain that for $t>0$,
\begin{align*}
\varphi'&=2f'\Bigl(f+\frac{A_{2}}{A_{0}}\,f''\Bigr)+\Bigl(\frac{A_{2}}{A_{0}}\Bigr)'f'^{2}=
\Bigl(\Bigl(\frac{A_{2}}{A_{0}}\Bigr)'-2\,\frac{A_{1}}{A_{0}}\Bigr)f'^{2}\\
&=-\frac{2((2\alpha+1)t^2+8(\alpha+2)^2)}{t^3}\,f'^{2}<0.
\end{align*}
Thus,  $\varphi$ is a decreasing function on $[0,\infty)$.
The claim \eqref{5-22} then follows since $$\varphi(q_{\al+1,k})=f^{2}(q_{\al+1,k}),\quad k=1,2,\cdots.$$

Thus,  to complete the proof of the lemma,  it suffices to show that for each $\alpha\ge -0.272$,  \begin{equation}\label{5-27}
|f(q_{1})|\ge
|f(q_{\al+3,1})|.
\end{equation}

For simplicity,  we write $q_1'=q_{\al+3,1}$.
Using \eqref{2-6-00}, and  by straightforward calculations, we obtain
\[
|f(q_{1})|=\frac{q_{1}^{2}j_{\alpha+2}(q_{1})}{4(\alpha+1)(\alpha+2)},\quad
|f(q_{1}')|=
-\frac{(q_{1}'^{2}+4(\alpha+2))j_{\alpha+2}(q_{1}')}{4(\alpha+1)(\alpha+2)}.
\]
Thus,
\begin{equation}\label{5-28-0}
\rho(\alpha):=\frac{|f(q_{1}')|}{|f(q_{1})|}=
-\frac{(q_{1}'^{2}+4(\alpha+2))j_{\alpha+2}(q_{1}')}{q_{1}^{2}j_{\alpha+2}(q_{1})}.
\end{equation}
To complete  the proof of \eqref{5-27}, it remains   to verify that  $\rho(\al) \leq 1$ for $\al\ge -0.272$.
We consider the following  two cases:   (i) $-0.272\leq \al\leq  0.575$,
(ii) $\al>0.575$.

For the first case, we use the fact that   $\rho(\alpha)$ as given in \eqref{5-28-0}  is an analytic function of $\alpha\ge -\frac12$.
Thus, for $\alpha\le 0.575$ we can use the very precise
approximation of $J_{\alpha}$ and $q_{\alpha,1}$  realized in  Maple to compute $\rho(\al)$.  Indeed, easy
numerical calculations shows  that $\alpha_{0}=-0.2729\cdots$ is a solution of the equation $\rho(\al)=1$, and the function $\rho(\alpha)$ is decreasing on $(-\f 12, 0.575]$,  and  $\rho(\alpha)\leq 1$  whenever $\al \in [\al_0, 0.575]$.

To estimate $\rho(\al)$ for the second case, we set $y(t):=t^{\frac12}J_{\alpha+2}(t)$, and express  $\rho(\al)$  as \[
\rho(\alpha)=
-\Bigl(1+\frac{4(\alpha+2)}{q_{1}'^{2}}\Bigr)
\Bigl(\frac{q_{1}}{q_{1}'}\Bigr)^{\alpha+\frac12}\,\frac{y(q_{1}')}{y(q_{1})}.
\]
As is well known, the function $y$ satisfies the differential equation
\begin{equation}\label{5-28}
y''+A(t)y=0\quad \text{with $A(t)=1-\frac{(\alpha+2)^{2}-1/4}{t^{2}}$}.
\end{equation}
Note that by \eqref{2-9-2}, $q_{1}>((\alpha+2)^{2}-1/4)^{\frac12}$. Thus, we have
$$ A(t)>0,\quad A'(t)>0,\quad \forall\,t\ge q_1.$$
As in  the proof of the claim \eqref{5-22}, the differential equation \eqref{5-28} allows us to construct  a  decreasing function on $[q_1, \infty)$. Indeed, let $$\psi(t):=y^{2}+\f 1 {A(t)} \Bl( \f {dy}{dt}\Br)^2.$$
Then
$$\psi'(t)=-\f {A'(t)}{A(t)^2} \Bl( \f {dy}{dt}\Br)^{2}<0,\quad t\ge q_1.$$   Since $q_1=q_{\al+1,1}<q_1'=q_{\al+3,1}$, it follows that  \begin{equation}\label{5-30-0}
\psi(q_{1}')<\psi(q_{1}).
\end{equation}
However, using the relations
$J_{\alpha+2}'(q_{1}')=\frac{\alpha+2}{q_{1}'}\,J_{\alpha+2}(q_{1}')$ and
$J_{\alpha+2}'(q_{1})=-\frac{\alpha+2}{q_{1}}\,J_{\alpha+2}(q_{1})$,  we obtain
\[
\psi(q_{1}')=y^{2}(q_{1}')\Bigl(1+\frac{(\alpha+\frac52)^{2}}
{q_{1}'^{2}-(\alpha+2)^{2}+\frac14}\Bigr),\quad
\psi(q_{1})=y^{2}(q_{1})\Bigl(1+\frac{(\alpha+\frac32)^{2}}
{q_{1}^{2}-(\alpha+2)^{2}+\frac14}\Bigr).
\]
Thus, using \eqref{5-30-0}, we can estimate the function $\rho(\al)$ as follows:
\begin{equation}\label{5-30}
\rho(\alpha)<\tilde{\rho}(\alpha):=\Bigl(1+\frac{4(\alpha+2)}{q_{1}'^{2}}\Bigr)
\Bigl(\frac{1+\frac{(\alpha+\frac32)^{2}}{q_{1}^{2}-(\alpha+2)^{2}+\frac14}}
{1+\frac{(\alpha+\frac52)^{2}}{q_{1}'^{2}-(\alpha+2)^{2}+\frac14}}\Bigr)^{\frac12}
\Bigl(\frac{q_{1}}{q_{1}'}\Bigr)^{\alpha+\frac12}.
\end{equation}
We then reduce to showing that $\tilde{\rho}(\al) \leq 1$ for $\al\ge 0.575$.  To this end, we use  the following uniform estimates on the first positive zeros of Bessel functions ( see
\cite{QW99}):
\begin{equation}\label{q1-bounds-0}
\alpha+c_{1}\alpha^{1/3}<q_{\alpha,1}<
\alpha+c_{1}\alpha^{1/3}+c_{2}\alpha^{-1/3},\quad \forall\,\al>0,
\end{equation}
where $c_{1}=1.855\cdots$ and $c_{2}=1.033\cdots$.
Substituting the  bounds \eqref{q1-bounds-0} into  the expression of  $\tilde{\rho}(\alpha)$ in \eqref{5-30},  one can easily verify via simple numerical calculations that $\tilde{\rho}(\alpha)<1$ for any $\alpha\ge 0.575$.
This completes the proof.

\section{Proofs of Theorem  \ref{thm-main-as} and Corollary \ref{cor-1-1}}\label{sec:6}
\begin{proof}[Proofs of Theorem  \ref{thm-main-as}]

The lower estimate in Theorem \ref{thm-main-as} follows directly from Theorem \ref{thm-1-2}, while the upper estimate in Theorem \ref{thm-main-as}  is an easy consequence of Corollary \ref{cor-L*} and Theorem \ref{thm-5-1}.
\end{proof}

\begin{proof}[ Proof of Corollary \ref{cor-1-1}]

By Theorem \ref{thm-main-as},  it suffices to show that
  \begin{equation}\label{6-1-0}
    a_{0}^{*}:=\frac{\int_{0}^{q_{1}}j_{\alpha+1}(t)t^{2\alpha+1}\,dt}
  {\int_{0}^{q_{1}}t^{2\alpha+1}\,dt}=
  \Bigl(\frac2e\Bigr)^{\alpha(1+O(\alpha^{-2/3}))},\quad \al\to\infty,
  \end{equation}
where $q_1=q_{\al+1,1}$.

To show \eqref{6-1-0},  we  use the following known  formula (see \cite[7.14.1 (7)]{BE53}):
  \begin{equation}\label{6-2-0}
    \int_{0}^{z}t^{2\alpha+1}j_{\alpha+1}(t)\,dt=
  2\alpha z^{\alpha+2}j_{\alpha+1}(z)S_{\alpha-1,\alpha}(z)-
  (2\alpha+2)z^{\alpha+1}j_{\alpha}(z)S_{\alpha,\alpha+1}(z),
  \end{equation}
  where $S_{\mu,\nu}(z)$ denotes the Lommel function.  We then obtain
  \[
  a_{0}^{*}=
  -(2\alpha+2)^{2}q_{1}^{-\alpha-1}j_{\alpha}(q_{1})S_{\alpha,\alpha+1}(q_{1}).
  \]
Note that by \eqref{2-9-2},
  $\alpha+1<q_{1}=\alpha+1+O((\alpha+1)^{1/3})$ as $\al\to\infty$.

  We will also  use the following estimate of the Lommel function:
  \begin{equation}\label{bessj-S-as}
  S_{\alpha,\alpha+1}(z)=z^{\alpha-1}(1+O(z^{-1})) \quad \text{uniformly for $z>\alpha$},\quad \text{as $\alpha\to \infty$}.
  \end{equation}
 Since we are unable to find this estimate in
  literature, we decide to include a proof here.  Indeed, using \cite[Th.~1.1]{Ne15}, we have that for an integer $N>\al$,
  \begin{equation}\label{bessj-S}
  S_{\alpha,\alpha+1}(z)=z^{\alpha-1}\Bigl (\sum_{k=0}^{N-1}
  \frac{\prod_{v=1}^{k}(\alpha-v+1)v}{(z/2)^{2k}}+r_{N}(z)\Bigr),\quad z>0.
  \end{equation}
Here $r_N(z)\equiv 0$  if $\alpha$ is a nonnegtive
  integer, and $r_N(z)$ can be estimated by an integral of the
  Macdonald function otherwise:
  \begin{equation}\label{6-4}
    |r_{N}(z)|\le \frac{2^{\alpha+1}z^{-2N}}{\Gamma(-\alpha)}
  \int_{0}^{\infty}t^{2N-\alpha}K_{\alpha+1}(t)\,dt=
  \frac{\Gamma(N+1)\Gamma(N-\alpha)}{\Gamma(-\alpha)(z/2)^{2N}}.
    \end{equation}
  Letting $\alpha\to \infty$ and setting $N=[\alpha+2]$ in \eqref{6-4},  we obtain that for $z>\alpha$,
  \[
  r_{N}(z)=O((2/e)^{2\alpha}(z/2)^{-2(N-\alpha)})=O(z^{-1}).
  \]
  Since each fraction in \eqref{bessj-S} is decreasing in $k\ge 1$,
  \eqref{bessj-S-as} then follows.

  Now using \eqref{6-2-0} and \eqref{bessj-S-as}, we obtain that
  for $\alpha\to \infty$
  \begin{align*}
  a_{0}^{*}&= -\frac{(2\alpha+2)^{2}}{q_{1}^{2}}\,j_{\alpha}(q_{1})(1+O(q_{1}^{-1}))\notag\\
  &\sim
  -4j_{\alpha}(q_{1})=
  -\frac{2^{\alpha+2}\Gamma(\alpha+1)}{q_{1}^{\alpha}}\,J_{\alpha+1}'(q_{1}),
  \end{align*}
  where the last step uses  the formula $J_{\alpha}(q_{1})=J_{\alpha+1}'(q_{1})$.  On the other hand, however, according to
  \cite[Eq.~10.19.12]{OLBC10}, we have
  \[
  -J_{\alpha+1}'(q_{1})=-J_{\alpha+1}'(\alpha+1+O((\alpha+1)^{1/3}))=
  c\alpha^{-2/3}+O(\alpha^{-4/3}),
  \]
  where  $c$ is a positive constant independent of $\al$ that can be  expressed explicitly  in terms of the Airy
  function.
  Thus, using Stirling's formula $\Gamma(\alpha+1)\sim
  (2\pi\alpha)^{1/2}(\alpha/e)^{\alpha}$, we obtain
  \[
  a_{0}^{*}\sim
  C\alpha^{-1/6}(2\alpha/e)^{\alpha}(\alpha+O(\alpha^{1/3}))^{-\alpha}=
  (2/e)^{\alpha(1+O(\alpha^{-2/3}))}\quad \text{as}\ \alpha\to \infty.
  \]
  This proves \eqref{6-1-0} and hence completes the proof of the corollary.
\end{proof}

\section{Applications in the  Remez-type problem}\label{sec-remez}

{ In this section,  we give an application of our results on Nikolskii constants in a Remez type problem, which appears frequently  in
approximation theory and number theory.

Consider a Lebesgue-measurable set   $E\subset \mathbb{R}^{d}$ for which there exists function  $f\in
\mathcal{E}_{1}^{d}\setminus \{0\}$ such that
\begin{equation}\label{7-1}
\int_{E}|f(x)|\,dx\ge \frac12\int_{\mathbb{R}^{d}}|f(x)|\,dx.
\end{equation}

We define the Remez constant  $\al_d^\ast$ to be the infimum of the Lebesgue measure $|E|$ over all measurable $E\subset \R^d$ with the above mentioned property.
In the case of $d=1$, the exact value of the constant $\al_d^\ast$ was founded in   \cite{MR14}, where  it  was proved that
$\alpha_{1}^{*}=\pi$ and the corresponding extremal function
is $\frac{\cos x}{1-(2x/\pi
)^{2}}$.  The exact  value of $\al_d^\ast$ for $d\ge 2$ remains unknown. Note that the Remez constant plays an important role in $L_1$-approximation of  functions with small support and, in particular, in the study of sparse representations (compressed sensing).

Using our results on Nikolskii constants, we may give an asymptotic estimate of the constant $\al_d^\ast$ as $d\to \infty$. To be precise, we first recall that
$\mathcal{N}(\RR^d)_{1}=\mathcal{N}(\RR^d)_{1,\infty}=
\frac{V_{d}}{(2\pi)^{d}}\,\mathcal{L}^\ast (d)$, the constant  $\mathcal{I}_{\alpha}$ is given in  \eqref{5-2-0},
$V_{d}=\frac{\pi^{d/2}}{\Gamma(d/2+1)}=\text{Vol} (B_d)$, and
$q_{\alpha,1}$ is the first positive zero of the Bessel function $J_\al$.
We will need the  following known result:

\begin{thm}[\cite{Go00}]\label{thm-7-1}
For $d\ge 1$, we have
\begin{align*}
r_{d}:&=\inf \Bl\{r>0:\   \  \text{$\exists $ $f\in\mathcal{E}_1^d\setminus \{0\}$ such that $\wh{f}(0)\ge 0$ and $f(x)\leq 0$ for all $|x|\ge r$}  \Br\}\\
&=2q_{\f d2-1,1}
\end{align*}
with the extremal function given by
\begin{equation}\label{fd}
f_{d}(x):=\frac{(j_{d/2-1}(|x|/2))^{2}}{1-(|x|/r_{d})^{2})}.
\end{equation}
\end{thm}

As a consequence, we have

\begin{cor}\label{cor-remez}
For $d\ge 2$,
\begin{equation}\label{7-3-0}
\frac{(2\pi)^{d}}{2V_{d}\mathcal{I}_{d/2-1}}\le \alpha_{d}^{*}\le
(2q_{d/2-1,1})^{d}V_{d}.
\end{equation}

In particular, we have
\begin{equation}\label{7-3}
(\sqrt{e/2})^{1+o(1)}\le \Bl(\frac{V_{d}\alpha_{d}^{*}}{(2\pi)^{d}}\Br)^{1/d}\le
e^{1+o(1)}\quad \text{as $d\to \infty$}.
\end{equation}

\end{cor}

Note that the lower estimate here  improves significantly the estimate
$\bl(\frac{V_{d}\alpha_{d}^{*}}{(2\pi)^{d}}\br)^{1/d}\ge 1$ given in
\cite{BKP12}.

\begin{proof}
  Let $E\subset \RR^d$ and $f\in\mathcal{E}_1^d\setminus \{0\}$ be such that \eqref{7-1} is satisfied. Then
\[
\frac12\int_{\mathbb{R}^{d}}|f(x)|\,dx\le \int_{E}|f(x)|\,dx\le |E|\,\|f\|_{\infty}\le
|E|\mathcal{N}(\RR^d)_{1,\infty}\|f\|_{1}.
\]
This in particular  implies that
\[
\alpha_{d}^{*}\ge \frac{1}{2\mathcal{N}(\RR^d)_{1}},
\]
which further implies the lower estimate $\al_d^\ast \ge \frac{(2\pi)^{d}}{2V_{d}\mathcal{I}_{d/2-1}}$  because
\[
\frac{(2\pi)^d}{V_d}\,\mathcal{N}(\RR^d)_{1}= \mathcal{L}^\ast(d)\le
\mathcal{I}_{d/2-1}.
\]
To show the corresponding upper estimate, let  $f_{d}$ be the function given  in \eqref{fd}.  Then
\[
0=\int_{\mathbb{R}^{d}}f_{d}(x)\,dx=\int_{|x|\le r_{d}}|f_{d}(x)|\,dx-
\int_{|x|\ge r_{d}}|f_{d}(x)|\,dx,
\]
and hence,
\[
\int_{|x|\le
r_{d}}|f_{d}(x)|\,dx=\frac12\int_{\mathbb{R}^{d}}|f_{d}(x)|\,dx.
\]
This  together with \eqref{thm-7-1} implies the upper estimate:
\[
\alpha_{d}^{*}\le (r_{d})^{d}V_{d}\leq (2q_{\f d2-1, 1})^d V_d.
\]

Finally, we prove \eqref{7-3}.  Note that the lower asymptotic estimate in \eqref{7-3}  follows  directly from Corollary \ref{cor-1-1}.  The upper estimate as  $d\to \infty$ follows from the upper estimate in \eqref{7-3-0} since
\[
\frac{V_{d}^{2}(r_{d})^{d}}{(2\pi)^{d}}=
\frac{\pi^{d}(2q_{d/2-1,1})^{d}}{(2\pi)^{d}\Gamma^{2}(d/2+1)}
\sim \frac{(d/2)^{d}}{(d/(2e))^{d}}=e^{d}.
\]
\end{proof}

}

\end{document}